\documentclass[final]{siamltex}
\usepackage{epsfig}
\usepackage{color}
\usepackage{amscd}
\usepackage{amsmath, amsfonts, amssymb, mathrsfs}
\usepackage{multirow}
\usepackage{enumerate}
\usepackage{verbatim}
\usepackage{cite}
\usepackage{tikz}
\usetikzlibrary{matrix,arrows}
\usepackage[justification=raggedright]{caption}
\usepackage{subcaption}
\usepackage[ruled]{algorithm2e}
\usepackage[export]{adjustbox}
\usepackage{wasysym}

\numberwithin{equation}{section}

\def \vec#1{{\bf{#1}}}

\makeatletter
\def\hlinewd#1{%
\noalign{\ifnum0=`}\fi\hrule \@height #1 %
\futurelet\reserved@a\@xhline}
\makeatother

\newcommand{\bi}{\begin{itemize}}
\newcommand{\ei}{\end{itemize}}

\newcommand{\diverg}{\vec{\nabla}\cdot}
\newcommand{\director}{\vec{n}}
\newcommand{\curl}{\vec{\nabla}\times}
\newcommand{\Ltwoinner}[3]{\langle #1,#2 \rangle_0}
\newcommand{\Ltwonorm}[2]{\Vert #1 \Vert_0}

\newcommand{\Ltwoinnerndim}[4]{\langle #1,#2 \rangle_0}

\newcommand{\Rone}{\mathbb{R}}

\newcommand{\diff}[1]{\, d#1}

\newcommand{\Hone}[1]{H^1(#1)}

\newcommand{\Honenot}[1]{H^1_0({#1})}

\newcommand{\Ltwo}[1]{L^2(#1)}

\newcommand{\Lp}[1]{L^p (\Omega)}

\newcommand{\Linfinity}[1]{L^{\infty}(\Omega)}

\newcommand{\triangulation}{\mathcal{T}_h}

\newcommand{\diam}{\text{diam }}

\allowdisplaybreaks

\newcounter{casenum}

\title{Error Estimators and Marking Strategies for Electrically Coupled Liquid Crystal Systems\thanks{Revised \today}}

\author{D. B. Emerson\thanks{Department of Mathematics, Tufts University, Medford, MA 02155 (david.emerson@tufts.edu).}}

\begin{document}

\maketitle

\begin{abstract}
This paper derives a posteriori error estimators for the nonlinear first-order optimality conditions associated with the electrically and flexoelectrically coupled Frank-Oseen model of liquid crystals, building on the results of \cite{Emerson6} for elastic systems. Estimators are proposed for both Lagrangian and penalty approaches to imposing the unit-length constraint required by the model. Moreover, theory is proven establishing the penalty method estimator as a reliable estimate of global approximation error and an efficient measure of local error, suitable for use in adaptive refinement. Numerical experiments conducted herein demonstrate significant improvements in both accuracy and efficiency with adaptive refinement guided by the proposed estimators for both constraint formulations. The numerical results also extend the simulations of \cite{Emerson6} to include systems with known analytical solutions, confirming the theoretical results and enabling performance comparisons for a selection of established marking strategies. In each case, the adapted grids successfully yield substantial reductions in computational work, comparable or better physical properties, and deliver more uniformly distributed error.
\end{abstract}

\begin{keywords}
liquid crystal simulation, coupled systems, a posteriori error estimators, adaptive mesh refinement
\end{keywords}

\begin{AMS}
76A15, 65N30, 49M15, 65N22, 65N55
\end{AMS}
 
\pagestyle{myheadings}
\thispagestyle{plain}
\markboth{\sc Emerson}{\sc AMR for Coupled Liquid Crystals Systems}


\section{Introduction}

As materials possessing mesophases with characteristics of both liquids and organized solids, liquid crystals exhibit many interesting physical properties inspiring extensive study and a wide range of applications. In addition to considerable use in modern display technologies, liquid crystals are used for nanoparticle organization \cite{Lagerwall1},  manufacture of structured nanoporous solids \cite{Wan1}, and efficient conversion of mechanical strain to electrical energy \cite{Harden1}, among many others.

The focus of this paper is nematic liquid crystals, which are rod-like molecules with long-range orientational order described by a vector field $\director(x, y, z) = (n_1, n_2, n_3)^T$, know as the director. For the model considered here, $\director$ is constrained to unit-length pointwise throughout the domain, $\Omega$. In addition to their elastic properties, liquid crystals are dielectrically active such that their structures are affected by the presence of electric fields. In addition, certain types of liquid crystals demonstrate flexoelectric coupling wherein deformations of the director produce internally generated electric fields \cite{Meyer1}. Thorough treatments of liquid crystal physics are found in \cite{Stewart1, Virga1}.

With the combination of highly coupled physical phenomena and complicated experimental behavior, numerical simulations of liquid crystal structures are fundamental to the study of novel physical phenomena \cite{Emerson5, Clerc1, RojasGomez1}, optimization of device design, and analysis of experimental observations. As many applications and experiments require simulations on two-dimensional (2D) and three-dimensional (3D) domains with complicated boundary conditions, the development of highly efficient and accurate numerical approaches is of significant importance. Effective a posteriori error estimators significantly increase the efficiency of numerical methods for partial differential equations (PDEs) and variational systems by guiding the construction of optimal discretizations via adaptive refinement. A wealth of research exists for the design and theoretical support of effective error estimators in the context of finite-element methods. This includes techniques treating both linear and nonlinear PDEs across a number of applications \cite{John4, Oden1, Verfurth3, Bank1, Babuska2}.

In \cite{Emerson6}, a reliable a posterior error estimator was developed for the first-order optimality conditions arising from minimization of the Frank-Oseen elastic free-energy model. Using the estimator to guide adaptive mesh refinement (AMR) in numerical simulations produced competitive solutions in terms of constraint conformance and free energy with considerably less computational work. However, no analytical error studies were performed at that time to confirm the theoretical bounds or definitively quantify efficiency gains. As such, the goal of this work is two-fold. First, we investigate the performance of the elastic error estimator for problems with known analytical solutions on both 2D and 3D domains, verifying the reliability theory of \cite{Emerson6}. Additionally, the known solutions enable a comparison of different marking strategies within the AMR framework, including techniques proposed in \cite{Dorfler1, Gui1, Gui2, Gui3}, which indicate that a well-chosen marking scheme yields even better efficiency. Second, we extend the elastic error estimator to consider systems with electric and flexoelectric coupling. The proposed, coupled, a posteriori error estimator is shown to be a reliable estimate of global approximation error and an efficient indicator of local error. Numerical experiments with both external and flexoelectrically induced electric fields demonstrate the performance of the estimator compared with uniform refinement.

This paper is organized as follows. In Section \ref{model}, the coupled Frank-Oseen free-energy model and associated variational systems for the first-order optimality conditions are introduced. Additional notation and prerequisite theoretical results to be applied in the reliability and efficiency proofs are discussed in Section \ref{preliminaries}. In Section \ref{theory}, the coupled error estimators are derived for both the penalty and Lagrangian formulations of the variational systems. In addition, proofs of reliability and efficiency for the penalty method estimator are constructed. The applied marking strategies are discussed in Section \ref{numerics}, and a set of numerical experiments is presented investigating the performance of the elastic and coupled error estimators. Finally, Section \ref{conclusions} provides some concluding remarks and a discussion of future work.

\section{Free-Energy Model and Optimality Conditions} \label{model}

Liquid crystals are simulated using a number of different models \cite{Davis1, Onsager1, Gartland1}. Here, we consider the Frank-Oseen free-energy model where, for a domain $\Omega$, the coupled equilibrium free energy is composed of three main components associated with elastic deformations, external electric fields, and flexoelectrically generated fields. Let $K_i \geq 0$, $i=1,2, 3$ be Frank constants. Assuming that each $K_i \neq 0$, define the tensor $\vec{Z} = \vec{I} - (1-\kappa) \director \otimes \director$, where $\kappa = K_2/K_3$. The Frank constants depend on the physical characteristics of the liquid crystal and have a significant impact on orientational structure \cite{Atherton2, Lee1}. 

We denote the classical $\Ltwo{\Omega}$ inner product and norm as $\Ltwoinner{\cdot}{\cdot}{\Omega}$ and $\Vert \cdot \Vert_0$, respectively, for both scalar and vector quantities. The coupled free-energy functional is then written
\begin{align}
\mathcal{G}(\director, \phi) &= \frac{1}{2}K_1 \Ltwonorm{\diverg \director}{\Omega}^2 + \frac{1}{2}K_3\Ltwoinnerndim{\vec{Z} \curl \director}{\curl \director}{\Omega}{3}  - \frac{1}{2}\epsilon_0\epsilon_{\perp}\Ltwoinnerndim{\nabla \phi}{\nabla \phi}{\Omega}{3} \nonumber \\
&\hspace{0.35cm} - \frac{1}{2}\epsilon_0 \epsilon_a \Ltwoinner{\director \cdot \nabla \phi}{\director \cdot \nabla \phi}{\Omega} + e_s \Ltwoinner{\diverg \director}{\director \cdot \nabla \phi}{\Omega} + e_b\Ltwoinnerndim{\director \times \curl \director}{\nabla \phi}{\Omega}{3}. \label{flexofunctional}
\end{align}
For a full derivation of the functional in \eqref{flexofunctional}, see \cite{Emerson2, Emerson5}. Throughout this paper, we assume the presence of Dirichlet boundary conditions, therefore the functional has been simplified using the null Lagrangian discussed in \cite{Stewart1}. Moreover, the free-energy expression has been non-dimensionalized using the approach detailed in \cite{Emerson3}.

The variable $\phi$ in \eqref{flexofunctional} denotes the electric potential and $\epsilon_0 > 0$ is the permittivity of free space. The dielectric anisotropy of the liquid crystal is $\epsilon_a = \epsilon_{\parallel} - \epsilon_{\perp}$, with the constants $\epsilon_{\parallel}, \epsilon_{\perp} > 0$ representing the parallel and perpendicular dielectric permittivity, respectively. For positive $\epsilon_a$, the director prefers parallel alignment with the electric field, while negative anisotropy indicates a perpendicular preference. Finally, $e_s$ and $e_b$ are material constants specifying the liquid crystal's flexoelectric response. Liquid crystal equilibrium states correspond to configurations that minimize the functional in \eqref{flexofunctional} subject to the local unit-length constraint, $\director \cdot \director -1$, on $\Omega$. Additionally, the relevant Maxwell's equations for a static electric field, $\diverg \vec{D} = 0$ and $\curl \vec{E} = \vec{0}$, known as Gauss' and Faraday's laws, respectively, must be satisfied. For this system, 
\begin{align} \label{ElectricDisplacement}
\vec{D} = -\epsilon_0 \epsilon_{\perp} \nabla \phi - \epsilon_0 \epsilon_a (\director \cdot \nabla \phi) \director + e_s \director (\diverg \director) + e_b (\director \times \curl \director).
\end{align}
Note that the use of an electric potential implies that Faraday's law is trivially satisfied, and it is straightforward to show that a minimizing pair, $(\director_*, \phi_*)$, adhering to the unit-length constraint, satisfies Gauss' law in weak form.

Throughout this paper, it is assumed that $\director \in H_{\vec{g}_1}^1(\Omega)^3 = \{ \vec{v} \in H^1(\Omega)^3 : \vec{v} = \vec{g}_1 \text{ on } \partial \Omega \}$ and $\phi \in H_{g_2}^1(\Omega) = \{ \psi \in H^1(\Omega) : \psi = g_2 \text{ on } \partial \Omega \}$, where $H^1(\Omega)$ denotes the classical Sobolev space with norm $\Vert \cdot \Vert_1$. The boundary functions $\vec{g}_1$ and $g_2$ are assumed to satisfy appropriate compatibility conditions for the domain. Note that if $\vec{g}_1 = \vec{0}$, the space $H^1_{\vec{g}_1}(\Omega)^3 = H^1_0(\Omega)^3$.

In order to enforce the pointwise unit-length constraint, we consider the penalty and Lagrange multiplier approaches studied in \cite{Emerson2, Emerson3}. The penalty method adds a weighted, positive term to the free-energy functional, penalizing deviation from the constraint such that for $\zeta > 0$
\begin{align*}
\mathcal{H}(\director, \phi) = \mathcal{G}(\director, \phi) + \frac{1}{2}\zeta \Ltwoinner{\director \cdot \director -1}{\director \cdot \director -1}{\Omega}.
\end{align*}
Taking the first variation of $\mathcal{H}(\director, \phi)$, the first-order optimality conditions are written
\begin{align}
\mathcal{P}(\director, \phi) = \mathcal{C}(\director, \phi) + 2\zeta \Ltwoinner{\vec{v} \cdot \director}{\director \cdot \director -1}{\Omega}  = 0 && \forall (\vec{v}, \psi) \in H^1_0(\Omega)^3 \times H^1_0(\Omega), \label{PenaltyFOOC}
\end{align}
where
\begin{align*}
&\mathcal{C}(\director, \phi) = K_1\Ltwoinner{\diverg \director}{\diverg \vec{v}}{\Omega} + K_3\Ltwoinnerndim{\vec{Z} \curl \director}{\curl \vec{v}}{\Omega}{3} - \epsilon_0 \epsilon_a \Ltwoinner{\director \cdot \nabla \phi}{\vec{v} \cdot \nabla \phi}{\Omega} \nonumber \\
& \qquad + (K_2-K_3)\Ltwoinner{\director \cdot \curl \director}{\vec{v} \cdot \curl \director}{\Omega}  - \epsilon_0 \epsilon_{\perp}\Ltwoinnerndim{\nabla \phi}{\nabla \psi}{\Omega}{3} - \epsilon_0 \epsilon_a \Ltwoinner{\director \cdot \nabla \phi}{\director \cdot \nabla \psi}{\Omega} \nonumber \\
& \qquad + e_s\big( \Ltwoinner{\diverg \director}{\vec{v} \cdot \nabla \phi}{\Omega} + \Ltwoinner{\diverg \vec{v}}{\director \cdot \nabla \phi}{\Omega} \big)  + e_b\big( \Ltwoinnerndim{\director \times \curl \vec{v}}{\nabla \phi}{\Omega}{3} \nonumber \\
&\qquad + \Ltwoinnerndim{\vec{v} \times \curl \director}{\nabla \phi}{\Omega}{3} \big) + e_s \Ltwoinner{\diverg \director}{\director \cdot \nabla \psi}{\Omega} +e_b  \Ltwoinnerndim{\director \times \curl \director}{\nabla \psi}{\Omega}{3}.
\end{align*}
Alternatively, the Lagrange multiplier approach uses a non-dimensionalized Lagrange multiplier to form the Lagrangian
\begin{align*}
\mathcal{L}(\director, \phi, \lambda) = \mathcal{G}(\director, \phi) + \frac{1}{2} \int_{\Omega} \lambda(\vec{x})((\director \cdot \director) - 1).
\end{align*}
The corresponding first-order optimality conditions are
\begin{align}
\mathcal{F}(\director, \phi, \lambda) &= \mathcal{C}(\director, \phi) + \int_{\Omega} \lambda (\director \cdot \vec{v}) \diff{V} + \int_{\Omega} \gamma ((\director \cdot \director) - 1) \diff{V} = 0 \label{lagrangeFOOC}
\end{align}
for all $(\vec{v}, \psi, \gamma) \in \Honenot{\Omega}^3 \times \Honenot{\Omega} \times \Ltwo{\Omega}$, where the constant coefficient of the last term has been absorbed into $\gamma$.

In \cite{Emerson6}, a posteriori error estimators were proposed for the first-order optimality conditions of purely elastic liquid crystal systems. Below, we extend those existing estimators to include electric and flexoelectric coupling for both constraint enforcement techniques. Moreover, we show that the penalty method estimator for the coupled systems is both reliable and locally efficient. While reliability and efficiency theory for the Lagrangian formulation remains under development, numerical results show that estimators for both constraint formulations perform well in practice.

\section{Preliminary Theory and Notation} \label{preliminaries}

In this section, some additional notation and requisite theoretical results used in subsequent sections are discussed. For the theory to follow, it is assumed that the domain $\Omega$ is open and connected, with a polyhedral boundary. For any open subset $\omega \subset \Omega$ with Lipschitz boundary, norms restricted to the subdomain are denoted with an index as $\Vert \cdot \Vert_{1, \omega}$ and $\Vert \cdot \Vert_{0, \omega}$. Let $\{\triangulation\}$, $0 < h \leq 1,$ be a quasi-uniform family of meshes subdividing $\Omega$ and satisfying the conditions
\begin{align}
&\max \{\diam T : T \in \triangulation\} \leq h \, \diam \Omega, \nonumber \\
&\min \{\diam B_T : T \in \triangulation\} \geq \rho h \, \diam \Omega \label{quasiuniform},
\end{align}
where $\rho > 0$ and $B_T$ is the largest ball contained in $T$ such that $T$ is star-shaped with respect to $B_T$. In addition, we assume that any triangulation satisfies the admissibility condition such that any two cells of $\triangulation$ are either disjoint or share a complete, smooth sub-manifold of their boundaries. For any $T \in \triangulation$, let $h_T = \diam T$, denote the set of edges of $T$ as $\mathcal{E}(T)$, and $h_E = \diam E$ for $E \in \mathcal{E}(T)$. It is also assumed that the mesh family is fine enough that $h_T, h_E \leq 1$. Note that the quasi-uniformity condition of \eqref{quasiuniform} ensures that the ratio $h_T/h_E$ is bounded above and below by constants independent of $h$, $T$, and $E$ and implies that the smallest angle of any $T$ is bounded from below by a constant independent of $h$ \cite{Verfurth2}.

The sets of vertices corresponding to $T$ and $E$ are written $\mathcal{N}(T)$ and $\mathcal{N}(E)$, respectively. The set of all edges for $\triangulation$ is written $\mathcal{E}_h = \bigcup_{T \in \triangulation} \mathcal{E}(T)$, and $\mathcal{E}_{h, \Omega}$ signifies the subset of interior edges. Finally, some specific subdomains of $\Omega$ are written
\begin{align*}
\omega_T &= \bigcup_{\mathcal{E}(T) \cap \mathcal{E}(T') \neq \emptyset} T', & \omega_E &= \bigcup_{E \in \mathcal{E}(T')} T',\\
\tilde{\omega}_T &= \bigcup_{\mathcal{N}(T) \cap \mathcal{N}(T') \neq \emptyset} T', & \tilde{\omega}_E &= \bigcup_{\mathcal{N}(E) \cap \mathcal{N}(T') \neq \emptyset} T'.
\end{align*}
For the triangulations, define a fixed reference element $\hat{T}$ and reference edge $\hat{E}$ as $\hat{T} = \{\hat{x} \in \mathbb{R}^n : \sum_{i = 1}^n \hat{x}_i \leq 1, \hat{x}_j \geq 0, 1 \leq j \leq n \}$ and $\hat{E} = \hat{T} \cap \{ \hat{x} \in \mathbb{R}^n : \hat{x}_n = 0 \}$. The triangulation is assumed to be affine equivalent such that, for any $T \in \triangulation$, there exists an invertible affine mapping from the reference components to $T$. For any $E \in \mathcal{E}_h$, we assign a unit normal vector $\eta_E$ coinciding with the outward normal for $E$ on the boundary. Then, for any piecewise continuous function $\psi$, the jump across $E$ in the direction $\eta_E$ is denoted as $[\psi]_E$. Finally, for $k \in \mathbb{N}$, define the finite-dimensional space 
\begin{align*}
S_h^{k, 0} &= \{\psi: \Omega \rightarrow \Rone : \psi \vert_T \in \Pi_k,  \forall T \in \triangulation \} \cap C(\bar{\Omega})
\end{align*}
where $\Pi_k$ is the set of polynomials of degree at most $k$, $\psi \vert_T$ is the restriction of $\psi$ to the element $T$, and $C(\bar{\Omega})$ is the collection of continuous functions on the closure of $\Omega$.

Making use of the notation and assumptions established above, a collection of important supporting theoretical results is gathered in this section and referenced in the efficiency and reliability theory developed in Section \ref{theory}. Let $I_h: L^1(\Omega) \rightarrow S_h^{1, 0}$ denote the Cl\'{e}ment interpolation operator \cite{Clement1, Verfurth1}. Then the following approximation error bound holds for $\triangulation$.

\begin{lemma} \label{clementlemma}
For any $T \in \triangulation$ and $E \in \mathcal{E}_h$
\begin{align*}
\Vert \psi - I_h \psi \Vert_{0, T} &\leq C_1 h_T \Vert \psi \Vert_{1, \tilde{\omega}_T} & \forall \psi \in \Hone{\tilde{\omega}_T}, \\
\Vert \psi - I_h \psi \Vert_{0, E} &\leq C_2 h_E^{1/2} \Vert \psi \Vert_{1, \tilde{\omega}_E} & \forall \psi \in \Hone{\tilde{\omega}_E},
\end{align*}
where $C_1$ and $C_2$ depend only on the quasi-uniformity condition in \eqref{quasiuniform}. 
\end{lemma}

Following the notation in \cite{Verfurth1, Verfurth2}, let $\Psi_{\hat{T}}, \Psi_{\hat{E}} \in C^{\infty}(\hat{T}, \Rone)$ be cut-off functions defined on the reference components $\hat{T}$ and $\hat{E}$ such that
\begin{align*}
&0 \leq \Psi_{\hat{T}} \leq 1, \quad \max_{\hat{x} \in \hat{T}} \Psi_{\hat{T}}(\hat{x}) = 1, \quad \Psi_{\hat{T}} = 0 \text{ on } \partial \hat{T}, \\
&0 \leq \Psi_{\hat{E}} \leq 1, \quad \max_{\hat{x} \in \hat{E}} \Psi_{\hat{E}}(\hat{x}) = 1, \quad \Psi_{\hat{E}} = 0 \text{ on } \partial \hat{T} \backslash \hat{E}.
\end{align*}
Define a continuation operator $\hat{P}: L^{\infty}(\hat{E}) \rightarrow L^{\infty}(\hat{T})$ as 
\begin{align*}
\hat{P} \hat{u}(\hat{x}_1, \ldots, \hat{x}_n) := \hat{u}(\hat{x}_1, \ldots, \hat{x}_{n-1})
\end{align*}
for all $\hat{x} \in \hat{T}$, and fix two arbitrary finite-dimensional subspaces, $V_{\hat{T}} \subset L^{\infty}(\hat{T})$ and $V_{\hat{E}} \subset L^{\infty}(\hat{E})$. Applying the affine mappings from reference components, corresponding functions, $\Psi_T$ and $\Psi_E$, operator $P: L^{\infty}(E) \rightarrow L^{\infty}(T)$, and spaces $V_T$ and $V_E$ are defined for arbitrary $T \in \triangulation$ and $E \in \mathcal{E}_h$ with analogous properties. Thus, the following lemma and corollary hold, c.f. \cite{Verfurth1, Verfurth2, Brenner1}.
\begin{lemma} \label{cutoffinequalities}
There are constants $C_1, \ldots, C_7$ depending only on the finite-dimensional spaces $V_{\hat{T}}$ and $V_{\hat{E}}$, the functions $\Psi_{\hat{T}}$ and $\Psi_{\hat{E}}$, and  the quasi-uniform bound of \eqref{quasiuniform} such that for all $T \in \triangulation$, $E \in \mathcal{E}(T)$, $u \in V_T$, and $\sigma \in V_E$
\begin{align}
C_1 \Vert u \Vert_{0, T} &\leq \sup_{v \in V_T} \frac{\int_T u \Psi_T v \diff{V}}{\Vert v \Vert_{0, T}} \leq \Vert u \Vert_{0, T}, \label{cutoff1}\\
C_2 \Vert \sigma \Vert_{0, E} &\leq \sup_{\tau \in V_E} \frac{\int_E \sigma \Psi_E \tau \diff{S}}{\Vert \tau \Vert_{0, E}} \leq \Vert \sigma \Vert_{0, E}, \label{cutoff2}\\
C_3 h_T^{-1} \Vert \Psi_T u \Vert_{0, T} &\leq \Vert \nabla (\Psi_T u) \Vert_{0, T}  \leq C_4 h_T^{-1} \Vert \Psi_T u \Vert_{0, T}, \nonumber \\
C_5 h_T^{-1} \Vert \Psi_E P\sigma \Vert_{0, T} &\leq \Vert \nabla (\Psi_E P \sigma) \Vert_{0, T}  \leq C_6 h_T^{-1} \Vert \Psi_E P \sigma \Vert_{0, T}, \nonumber \\
\Vert \Psi_E P \sigma \Vert_{0, T} &\leq C_7 h_T^{1/2} \Vert \sigma \Vert_{0, E}. \label{cutoff5}
\end{align}
\end{lemma}
Note that with quasi-uniformity of the triangulation, after proper adjustment of $C_i$ in any of the above inequalities, the mesh constant $h_T$ may be exchanged for $h_E$ while maintaining the inequality.
\begin{corollary} \label{cutoffexpansion}
Under the assumptions of Lemma \ref{cutoffinequalities}, there exists a $\bar{C}_4 > 0$ and $\bar{C}_6 > 0$ such that
\begin{align}
\Vert \Psi_T u  \Vert_{1,T} &\leq \bar{C}_4 h_T^{-1} \Vert \Psi_T u \Vert_{0, T}, \label{expansion1} \\
\Vert \Psi_E P \sigma \Vert_{1, T} &\leq \bar{C}_6 h_T^{-1} \Vert \Psi_E P \sigma \Vert_{0, T}. \label{expansion2}
\end{align}
\end{corollary}

Finally, we state two key propositions from the framework developed by Verf\"{u}rth \cite{Verfurth1, Verfurth2}. Let $X$ and $Y$ be Banach spaces with norms $\Vert \cdot \Vert_X$ and $\Vert \cdot \Vert_Y$ and denote the space of continuous linear maps from $X$ to $Y$ as $\mathcal{L}(X, Y)$ with the natural operator norm $\Vert \cdot \Vert_{\mathcal{L}(X, Y)}$. The subset of linear homeomorphisms from $X$ to $Y$ is written $\text{Isom}(X, Y)$. Define $Y^* = \mathcal{L}(Y, \Rone)$ to be the dual space of $Y$, with norm $\Vert \cdot \Vert_{Y^*}$, where the associated duality pairing is written $\langle \cdot, \cdot \rangle$. Let $F \in C^1(X, Y^*)$ be a continuously differentiable function for which a solution $u \in X$ is sought such that $F(u) = 0$. Denoting the derivative of $F$ as $DF$ and a ball of radius $R > 0$ centered at $u \in X$ as $B(u, R) = \{ v \in X : \Vert u - v \Vert_X < R \}$, the first proposition is as follows.
\begin{proposition}[\hspace{-4pt} {\cite[Pg. 47]{Verfurth2}}] \label{nonlinearErrorEstimation}
Let $u_0 \in X$ be a regular solution to $F(u) = 0$ in the sense that $DF(u_0) \in \text{Isom}(X, Y^*)$. Assume that $DF$ is Lipschitz continuous at $u_0$, where there exists an $R_0 > 0$ such that
\begin{align*}
\gamma = \sup_{u \in B(u_0, R_0)} \frac{\Vert DF(u) - DF(u_0) \Vert_{\mathcal{L}(X, Y^*)}}{\Vert u - u_0 \Vert_X} < \infty.
\end{align*}
Set $R = \min \big \{ R_0, \gamma^{-1} \Vert DF(u_0)^{-1} \Vert^{-1}_{\mathcal{L}(Y^*, X)}, 2 \gamma^{-1} \Vert DF(u_0) \Vert_{\mathcal{L}(X, Y^*)} \big \}$.
Then the error estimate
\begin{align*}
\frac{1}{2} \Vert DF(u_0) \Vert_{\mathcal{L}(X, Y^*)}^{-1} \Vert F(u) \Vert_{Y^*} \leq \Vert u - u_0 \Vert_X \leq 2 \Vert DF(u_0)^{-1} \Vert_{\mathcal{L}(Y^*, X)} \Vert F(u) \Vert_{Y^*}
\end{align*}
holds for all $u \in B(u_0, R)$.
\end{proposition}

Let $X_h \subset X$ and $Y_h \subset Y$ be finite-dimensional subspaces and $F_h \in C(X_h, Y_h^*)$ be an approximation of $F$. Consider the discretized problem of finding $u_h \in X_h$ such that $F_h(u_h) = 0$.
\begin{proposition}[\hspace{-4pt} {\cite[Pg. 52]{Verfurth2}}] \label{auxiliarySpaceInequality}
Let $u_h \in X_h$ be an approximate solution to $F_h(u_h) = 0$ in the sense that $\Vert F_h(u_h) \Vert_{Y_h^*}$ is ``small.'' Assume that there is a restriction operator $R_h \in \mathcal{L}(Y, Y_h)$, a finite-dimensional space $\tilde{Y}_h \subset Y$, and an approximation $\tilde{F}_h : X_h \rightarrow Y^*$ of $F$ at $u_h$ such that 
\begin{align*}
\Vert (\text{Id}_Y - R_h)^*\tilde{F}_h(u_h) \Vert_{Y^*} \leq C_0 \Vert \tilde{F}_h (u_h) \Vert_{\tilde{Y}_h^*},
\end{align*}
where $\text{Id}_Y$ is the identity operator on $Y$, $^*$ indicates application of $(\text{Id}_Y - R_h)$ to the dual variables, and $C_0 > 0$ is independent of $h$. Then the following estimate holds.
\begin{align*}
\Vert F(u_h) \Vert_{Y^*} &\leq C_0 \Vert \tilde{F}_h (u_h) \Vert_{\tilde{Y}^*_h} + \Vert (\text{Id}_Y - R_h)^*[F(u_h) - \tilde{F}_h(u_h)] \Vert_{Y^*} \\
& \hspace{1em} + \Vert R_h \Vert_{\mathcal{L}(Y, Y_h)} \Vert F(u_h) - F_h(u_h) \Vert_{Y_h^*} + \Vert R_h \Vert_{\mathcal{L}(Y, Y_h)} \Vert F_h(u_h) \Vert_{Y_h^*}.
\end{align*}
\end{proposition}
The first result provides an approximation error bound using the residual, while the second yields a concrete set of terms bounding the residual from above. 

\section{Reliable and Efficient Coupled Error Estimators} \label{theory}

In this section, we propose a posteriori error estimators for the first-order optimality conditions of Section \ref{model}, extending the estimators of \cite{Emerson6} to include electric and flexoelectric coupling. Furthermore, using the theory outlined in the previous section, the estimator associated with the penalty method is shown to be a reliable estimate of global approximation error and an efficient indicator of local error, suitable for use in AMR schemes. 

To begin, consider the first-order optimality conditions for the penalty method in \eqref{PenaltyFOOC}. Let $Y = X_0 = \Honenot{\Omega}^3 \times \Honenot{\Omega}$ and $X = H_{\vec{g}_1}^1(\Omega)^3 \times H_{g_2}^1(\Omega)$. Then $\mathcal{P}(\director, \phi) \in C^1(X, Y^*)$, and the Dirichlet boundary conditions imply that for a fixed $(\director, \phi) \in X$, $D\mathcal{P}(\director, \phi): X_0 \rightarrow Y^*$. In discretizing the variational system, we consider general discrete spaces 
\begin{align*}
[S_h^{1, 0}]^3 \subset V_h \subset [S_h^{s, 0}]^3, & & [S_h^{1, 0}] \subset Q_h \subset [S_h^{t, 0}],
\end{align*}
for $s, t \geq 1$ and the finite-dimensional space $Y_h = \{ (v_h, \psi_h) \in V_h \times Q_h : \vec{v}_h = \vec{0} \text{ and } \psi_h = 0 \text{ on } \Gamma \}$. For the theory presented here, we assume that the imposed boundary conditions on $(\director, \phi)$ are exactly representable on the coarsest mesh of $\{\triangulation\}$. Observe that this assumption on the boundary conditions admits projection of the boundary functions $\vec{g}_1$ and $g_2$ onto the coarsest mesh. Thus, the analysis to follow concerns estimation of the error arising in solution approximations on the interior of $\Omega$ but not from approximation of the boundary conditions. Hence, set $X_h = (V_h \times Q_h \cap X)$. Note that in the numerical results below, any boundary condition functions are interpolated with mesh refinement.

For $(\vec{v}, \psi) \in Y$ and $\langle \mathcal{P}(\director, \phi), (\vec{v}, \psi) \rangle$ define the discrete approximation
\begin{align*}
\langle \mathcal{P}_h(\director_h, \phi_h), (\vec{v}_h, \psi_h) \rangle = \langle \mathcal{P}(\director_h, \phi_h), (\vec{v}_h, \psi_h) \rangle,
\end{align*}
for $(\director_h, \phi_h) \in X_h, (\vec{v}_h, \psi_h) \in Y_h$. For the remainder of this section, assume that the pair $(\director_h, \phi_h)$ is a solution to the discrete problem
\begin{align}
\mathcal{P}_h(\director_h, \phi_h)=0 && \forall (\vec{v}_h, \psi_h) \in Y_h. \label{discretePenaltyFOOC}
\end{align}
In order to simplify notation, we define the vector and scalar quantities
\begin{align*}
\vec{p} &= -K_1 \nabla (\diverg \director_h) + K_3 \curl (\vec{Z}(\director_h) \curl \director_h) + (K_2 - K_3)(\director_h \cdot \curl \director_h) \curl \director_h \\
& \qquad  + 2 \zeta ((\director_h \cdot \director_h - 1) \director_h) - \epsilon_0 \epsilon_a ((\director_h \cdot \nabla \phi_h) \nabla \phi_h) + e_s (\diverg \director_h)\nabla \phi_h \\
& \qquad  - e_s \nabla (\director_h \cdot \nabla \phi_h)  + e_b (\curl \director_h \times \nabla \phi_h) + e_b \curl (\nabla \phi_h \times \director_h), \\
q &= \epsilon_0 \epsilon_{\perp} \Delta \phi_h + \epsilon_0 \epsilon_a \diverg ((\director_h \cdot \nabla \phi_h) \director_h) - e_s \diverg ((\diverg \director_h) \director_h) - e_b \diverg (\director_h \times \curl \director_h), \\
\hat{\vec{p}} & =  [K_1(\diverg \director_h) \eta_E + K_3 (\vec{Z}(\director_h)\curl \director_h) \times \eta_E + e_s (\director_h \cdot \nabla \phi_h) \eta_E \\
& \qquad + e_b ((\nabla \phi_h \times \director_h) \times \eta_E) ]_E,  \\
\hat{q} &=[ - \epsilon_0 \epsilon_{\perp} (\nabla \phi_h \cdot \eta_E) - \epsilon_0 \epsilon_a (\director_h \cdot \nabla \phi_h)(\director_h \cdot \eta_E) + e_s ((\diverg \director_h) \director_h) \cdot \eta_E \\
& \qquad + e_b (\director_h \times \curl \director_h) \cdot \eta_E ]_E,
\end{align*}
where $E \in \mathcal{E}_{h, \Omega}$. Integrating $\langle \mathcal{P}(\director_h, \phi_h), (\vec{v}, \psi) \rangle$ by parts elementwise for each $T \in \triangulation$, using the fact that $\vec{v}_h$ and $\psi_h$ are zero on the boundary, and gathering terms as done in \cite[Eq. 17]{Emerson6} yields
\begin{align}
\langle \mathcal{P}(\director_h, \phi_h), (\vec{v}, \psi) \rangle &= \sum_{T \in \triangulation} \int_T \vec{p} \cdot \vec{v} \diff{V} + \int_T q \cdot \psi \diff{V} \nonumber \\
& \qquad + \sum_{E \in \mathcal{E}_{h,  \Omega}}  \int_E \hat{\vec{p}} \cdot \vec{v} \diff{S} + \int_E \hat{q} \cdot \psi \diff{S}. \label{penaltyIntegrationByParts}
\end{align}
This form suggests a local estimator,
\begin{align*}
\Theta_T &= \Bigg \{ h_T^2 \left( \Vert \vec{p} \Vert_{0, T}^2 + \Vert q \Vert_{0, T}^2 \right) + \sum_{E \in \mathcal{E}(T) \cap \mathcal{E}_{h, \Omega}} h_E \left ( \Vert \hat{\vec{p}} \Vert_{0, E}^2 + \Vert \hat{q} \Vert_{0, E}^2 \right) \Bigg \}^{1/2},
\end{align*}
for any $T \in \triangulation$. Note that if no external electric field or flexoelectric coupling is present, $\Theta_T$ collapses to the elastic estimator of \cite{Emerson6}. In addition, the quantity $\Vert q \Vert_{0, T}$ locally measures the solution's conformance to the strong form of Gauss' law.

Let $R_h: Y \rightarrow Y_h$ be a restriction operator such that $R_h(\vec{u}, \varphi) = (I_h u_1, I_h u_2, I_h u_3, I_h \varphi)$ where $I_h$ is the Cl\'{e}ment operator of Lemma \ref{clementlemma}. Further, as no forcing function or Neumann boundary conditions are present, set
\begin{align*}
\langle \tilde{\mathcal{P}}_h(\director_h, \phi_h), (\vec{v}, \psi) \rangle = \langle \mathcal{P}(\director_h, \phi_h), (\vec{v}, \psi) \rangle.
\end{align*}
This trivially implies that 
\begin{align}
\Vert (\text{Id}_Y - R_h)^*[\mathcal{P}(\director_h, \phi_h) - \tilde{\mathcal{P}}_h(\director_h, \phi_h)]_{Y^*} &= 0, \label{zeroComponent1} \\
\Vert \mathcal{P}(\director_h, \phi_h) - \mathcal{P}_h(\director_h, \phi_h) \Vert_{Y^*} &= 0. \label{zeroComponent2}
\end{align}
With the above definitions, the following lemma holds.
\begin{lemma} \label{restrictionUpperBound}
There exists a constant $C > 0$ independent of $h$ such that 
\begin{align*}
\Vert (\text{Id}_Y - R_h)^*\tilde{\mathcal{P}}_h(\director_h, \phi_h) \Vert_{Y^*} \leq C \left( \sum_{T \in \triangulation} \Theta_T^2 \right)^{1/2}.
\end{align*}
\end{lemma}
\begin{proof}
First note that 
\begin{align}
&\Vert (\text{Id}_Y - R_h)^*\tilde{\mathcal{P}}_h(\director_h, \phi_h) \Vert_{Y^*} \nonumber \\
& = \sup_{\substack{[\vec{v}, \psi] \in Y \\ \Vert [\vec{v}, \psi] \Vert_Y = 1}} \sum_{T \in \triangulation} \sum_{i=1}^3 \int_T p_i \cdot (v_i - I_h v_i) \diff{V} + \int_T q \cdot (\psi - I_h \psi) \diff{V} \nonumber \\ 
& \hspace{5em} + \sum_{E \in \mathcal{E}_{h, \Omega}} \sum_{i = 1}^3 \int_E \hat{p}_i \cdot (v_i - I_h v_i) \diff{S} + \int_E \hat{q} \cdot (\psi - I_h \psi) \diff{S} \nonumber \\
& \leq \sup_{\substack{[\vec{v}, \psi] \in Y \\ \Vert [\vec{v}, \psi] \Vert_Y = 1}} \sum_{T \in \triangulation} \sum_{i=1}^3 C_1 h_T \Vert p_i \Vert_{0, T} \Vert v_i \Vert_{1, \tilde{\omega}_T} + C_1 h_T \Vert q \Vert_{0, T} \Vert \psi \Vert_{1, \tilde{\omega}_T} \nonumber \\
& \hspace{5em} + \sum_{E \in \mathcal{E}_{h, \Omega}} \sum_{i = 1}^3 C_2 h_E^{1/2} \Vert \hat{p}_i \Vert_{0, E} \Vert v_i \Vert_{1, \tilde{\omega}_E} + C_2 h_E^{1/2} \Vert \hat{q} \Vert_{0, E} \Vert \psi \Vert_{1, \tilde{\omega}_E}, \label{CSClementInequality}
\end{align}
where \eqref{CSClementInequality} is given by applying the Cauchy-Schwarz inequality and Lemma \ref{clementlemma} to the interpolation quantities. Using the Cauchy-Schwarz inequality for sums and letting $\tilde{C} = \max(C_1, C_2)$ implies that
\begin{align*}
& \Vert (\text{Id}_Y - R_h)^*\tilde{\mathcal{P}}_h(\director_h, \phi_h) \Vert_{Y^*} \\
&\leq \sup_{\substack{[\vec{v}, \psi] \in Y \\ \Vert [\vec{v}, \psi] \Vert_Y = 1}}  \tilde{C} \Bigg ( \sum_{T \in \triangulation} h_T^2 \left( \Vert \vec{p} \Vert_{0, T}^2 + \Vert q \Vert_{0, T}^2 \right) + \sum_{E \in \mathcal{E}_{h, \Omega}} h_E \left( \Vert \hat{\vec{p}} \Vert_{0, E}^2 + \Vert \hat{q} \Vert_{0, E}^2 \right)\Bigg )^{1/2} \\
& \hspace{0.85in} \cdot \Bigg(\sum_{T \in \triangulation} \Vert \vec{v} \Vert_{1, \tilde{\omega}_T}^2 + \Vert \psi \Vert_{1, \tilde{\omega}_T}^2 +  \sum_{E \in \mathcal{E}_{h, \Omega}} \Vert \vec{v} \Vert_{1, \tilde{\omega}_E}^2 + \Vert \psi \Vert_{1, \tilde{\omega}_E}^2\Bigg)^{1/2}.
\end{align*}
Finally, there exists a constant $C_* > 0$ independent of $h$ taking into account repeated elements such that
\begin{align*}
\left (\sum_{T \in \triangulation} \Vert \vec{v} \Vert_{1, \tilde{\omega}_T}^2 + \Vert \psi \Vert_{1, \tilde{\omega}_T}^2 +  \sum_{E \in \mathcal{E}_{h, \Omega}} \Vert \vec{v} \Vert_{1, \tilde{\omega}_E}^2 + \Vert \psi \Vert_{1, \tilde{\omega}_E}^2 \right)^{1/2} \leq C_*  \left \Vert [\vec{v}, \psi] \right \Vert_Y.
\end{align*}
Hence,
\begin{align*}
&\Vert (\text{Id}_Y - R_h)^*\tilde{\mathcal{P}}_h(\director_h, \phi_h) \Vert_{Y^*} \\
& \leq \sup_{\substack{[\vec{v}, \psi] \in Y \\ \Vert [\vec{v}, \psi] \Vert_Y = 1}} C_* \tilde{C} \Vert [\vec{v}, \psi] \Vert_Y \Bigg( \sum_{T \in \triangulation} h_T^2 \left( \Vert \vec{p} \Vert_{0, T}^2 + \Vert q \Vert_{0, T}^2 \right) \nonumber \\
& \hspace{1.25in} + \sum_{E \in \mathcal{E}_{h, \Omega}} h_E \left( \Vert \hat{\vec{p}} \Vert_{0, E}^2 + \Vert \hat{q} \Vert_{0, E}^2 \right) \Bigg)^{1/2}  \leq C \left ( \sum_{T \in \triangulation} \Theta_T^2 \right)^{1/2}.
\end{align*}
The final inequality is obtained by simply noting that the jump components are summed over $E \in \mathcal{E}_{h, \Omega}$.
\end{proof}

Next, define the finite-dimensional auxiliary space $\tilde{Y}_h \subset Y$ as
\begin{align*}
\tilde{Y}_h &= \text{span} \{ [\Psi_T \vec{v}, 0], [\Psi_E P \sigma, 0], [\vec{0}, \Psi_T \psi], [\vec{0}, \Psi_E P \tau] \\
& \hspace{0.8in} : \vec{v} \in [\Pi_{k \vert_T}]^3, \sigma \in [\Pi_{k \vert_E}]^3, \psi \in \Pi_{l \vert_T}, \tau \in \Pi_{l \vert_E}, T \in \triangulation, E \in \mathcal{E}_{h, \Omega} \},
\end{align*}
where $k \geq \max(3s, s+2(t-1))$ and $l \geq 2s + (t-1)$. For this space, the following lemma holds.
\begin{lemma} \label{auxiliarySpaceUpperBound}
There exists a $C > 0$ independent of $h$ such that
\begin{align*}
\Vert \tilde{\mathcal{P}}_h(\director_h, \phi_h) \Vert_{\tilde{Y}_h^*} \leq C \left ( \sum_{T \in \triangulation} \Theta_T^2 \right)^{1/2}.
\end{align*}
\end{lemma}
\begin{proof}
Applying the Cauchy-Schwarz inequality implies that
\begin{align*}
\Vert \tilde{\mathcal{P}}_h(\director_h, \phi_h) \Vert_{\tilde{Y}_h^*} &= \sup_{\substack{[\vec{v_h}, \psi_h] \in \tilde{Y}_h \\ \Vert [\vec{v}_h, \psi_h] \Vert_Y = 1}} \sum_{T \in \triangulation} \int_T \vec{p} \cdot \vec{v}_h \diff{V} + \int_T q \cdot \psi_h \diff{V} \nonumber \\
& \hspace{1.5in} + \sum_{E \in \mathcal{E}_{h, \Omega}} \int_E \hat{\vec{p}} \cdot \vec{v}_h \diff{S} + \int_E \hat{q} \cdot \psi_h \diff{S} \\
& \leq \sup_{\substack{[\vec{v_h}, \psi_h] \in \tilde{Y}_h \\ \Vert [\vec{v}_h, \psi_h] \Vert_Y = 1}} \sum_{T \in \triangulation} \Vert \vec{p} \Vert_{0, T} \Vert \vec{v}_h \Vert_{0, T} + \Vert q \Vert_{0, T} \Vert \psi_h \Vert_{0, T} \\
& \hspace{1.5in} + \sum_{E \in \mathcal{E}_{h, \Omega}} \Vert \hat{\vec{p}} \Vert_{0, E} \Vert \vec{v}_h \Vert_{0, E} + \Vert \hat{q} \Vert_{0, E} \Vert \psi_h \Vert_{0, E}.
\end{align*}
Using the definition of $\tilde{Y}_h$, quasi-uniformity of the mesh, and standard finite-element scaling arguments implies that
\begin{align}
&\Vert \tilde{\mathcal{P}}_h(\director_h, \phi_h) \Vert_{\tilde{Y}_h^*} \nonumber \\
& \leq \sup_{\substack{[\vec{v_h}, \psi_h] \in \tilde{Y}_h \\ \Vert [\vec{v}_h, \psi_h] \Vert_Y = 1}} \sum_{T \in \triangulation} C_1 h_T \Vert \vec{p} \Vert_{0, T} \Vert \vec{v}_h \Vert_{1, T} + C_1 h_T \Vert q \Vert_{0, T} \Vert \psi_h \Vert_{1, T} \nonumber \\
& \hspace{0.9in} + \sum_{E \in \mathcal{E}_{h, \Omega}} C_2 h_E^{1/2} \Vert \hat{\vec{p}} \Vert_{0, E} \Vert \vec{v}_h \Vert_{1, \omega_E} + C_2 h_E^{1/2} \Vert \hat{q} \Vert_{0, E} \Vert \psi_h \Vert_{1, \omega_E} \nonumber \\
& \leq \sup_{\substack{[\vec{v_h}, \psi_h] \in \tilde{Y}_h \\ \Vert [\vec{v}_h, \psi_h] \Vert_Y = 1}} \tilde{C} \Bigg( \sum_{T \in \triangulation} h_T^2 \left( \Vert \vec{p} \Vert_{0, T}^2 + \Vert q \Vert_{0, T}^2 \right) + \sum_{E \in \mathcal{E}_{h, \Omega}} h_E \left( \Vert \hat{\vec{p}} \Vert_{0, E}^2 + \Vert \hat{q} \Vert_{0, E}^2 \right) \Bigg)^{1/2} \nonumber \\
& \hspace{0.9in} \cdot \Bigg(  \sum_{T \in \triangulation} \Vert \vec{v}_h \Vert_{1, T}^2 + \Vert \psi_h \Vert_{1, T}^2 + \sum_{E \in \mathcal{E}_{h, \Omega}} \Vert \vec{v}_h \Vert_{1, \omega_E}^2 + \Vert \psi_h \Vert_{1, \omega_E}^2 \Bigg)^{1/2}, \label{CSSumInequality}
\end{align}
where $\tilde{C} = \max(C_1, C_2)$ and \eqref{CSSumInequality} is given by the Cauchy-Schwarz inequality for sums. Note, as above, there exists a $C_* > 0$, independent of $h$ and taking into account repeated elements in each sum, such that
\begin{align*}
\Bigg(  \sum_{T \in \triangulation} \Vert \vec{v}_h \Vert_{1, T}^2 + \Vert \psi_h \Vert_{1, T}^2 + \sum_{E \in \mathcal{E}_{h, \Omega}} \Vert \vec{v}_h \Vert_{1, \omega_E}^2 + \Vert \psi_h \Vert_{1, \omega_E}^2 \Bigg)^{1/2} \leq C_* \Vert [\vec{v}_h, \psi_h] \Vert_Y.
\end{align*}
Applying the inequality above to \eqref{CSSumInequality} and using the fact that the supremum is taken over $\Vert [\vec{v}_h, \psi_h] \Vert_Y = 1$ implies that
\begin{align*}
&\Vert \tilde{\mathcal{P}}_h(\director_h, \phi_h) \Vert_{\tilde{Y}_h^*} \\
&\leq C_* \tilde{C} \Bigg( \sum_{T \in \triangulation} h_T^2 \left( \Vert \vec{p} \Vert_{0, T}^2 + \Vert q \Vert_{0, T}^2 \right ) + \sum_{E \in \mathcal{E}_{h, \Omega}} h_E \left( \Vert \hat{\vec{p}} \Vert_{0, E}^2 + \Vert \hat{q} \Vert_{0, E}^2 \right) \Bigg)^{1/2} \\
& \leq C \left ( \sum_{T \in \triangulation} \Theta_T^2 \right)^{1/2}.
\end{align*}
As in the previous proof, the last inequality makes use of the fact that the jump components are summed over $E \in \mathcal{E}_{h, \Omega}$.
\end{proof}

The final inequality required to demonstrate reliability of the error estimator is
\begin{align*}
\Vert (\text{Id}_Y - R_h)^*\tilde{\mathcal{P}}_h(\director_h, \phi_h) \Vert_{Y^*} \leq C \Vert \tilde{\mathcal{P}}_h(\director_h, \phi_h) \Vert_{\tilde{Y}_h^*},
\end{align*}
for $C > 0$ and independent of $h$. With the result of Lemma \ref{restrictionUpperBound}, it is sufficient to prove the next lemma.
\begin{lemma} \label{auxilliarySpaceLowerBound}
There exists a $C > 0$, independent of $h$ such that 
\begin{align*}
C \left ( \sum_{T \in \triangulation} \Theta_T^2 \right)^{1/2} \leq \Vert \tilde{\mathcal{P}}_h(\director_h, \phi_h) \Vert_{\tilde{Y}_h^*}.
\end{align*}
\end{lemma}

{\em Proof}.
Fix an arbitrary $T \in \triangulation$ and an edge $E \in \mathcal{E}(T) \cap \mathcal{E}_{h, \Omega}$. Further, define the restricted space $\tilde{Y}_{h \vert \omega}$, for $\omega \in \{T, \omega_E, \omega_T \}$, as the set of functions $\vec{f} \in \tilde{Y}_h$ with $\text{supp}(\vec{f}) \subset \omega$. Finally, denote the product spaces $ \left( [\Pi_{k \vert T}]^3 \times \Pi_{l \vert T} \right ) \backslash \{(\vec{0}, 0) \}$ and $\left ( [\Pi_{k \vert E}]^3 \times \Pi_{l \vert E} \right ) \backslash \{(\vec{0}, 0) \}$ as $\Pi_{k, l, T}$, $\Pi_{k, l, E}$, respectively. Note that the constants in this proof correspond to those of Lemma \ref{cutoffinequalities} or Corollary \ref{cutoffexpansion}. First, consider
\begin{align}
&C_1 \bar{C}_4 ^{-1} h_T \Vert [\vec{p}, q] \Vert_{0, T} \nonumber \\
& \qquad \leq \sup_{[\vec{w}, u] \in \Pi_{k, l, T}} \bar{C}_4^{-1} h_T \Vert [\Psi_T \vec{w}, \Psi_T u] \Vert_{0, T}^{-1} \int_T (\vec{p}, q) \cdot (\Psi_T \vec{w}, \Psi_T u) \diff{V} \label{C1SupInequality} \\
& \qquad \leq \sup_{[\vec{w}, u] \in \Pi_{k, l, T}} \Vert [\Psi_T \vec{w}, \Psi_T u] \Vert_{1, T}^{-1} \int_T (\vec{p}, q) \cdot (\Psi_T \vec{w}, \Psi_T u) \diff{V}. \label{C4PointcareInequality}
\end{align}
The inequality in \eqref{C1SupInequality} is given by applying \eqref{cutoff1} of Lemma \ref{cutoffinequalities}, while the subsequent inequality in \eqref{C4PointcareInequality} relies on \eqref{expansion1} of Corollary \ref{cutoffexpansion}. Noting that both $\Psi_T \vec{w}$ and $\Psi_T u$ vanish at the boundary of $T$,
\begin{align}
C_1 \bar{C}_4 ^{-1} h_T \Vert [\vec{p}, q] \Vert_{0, T} &\leq \sup_{[\vec{w}, u] \in \Pi_{k, l, T}} \Vert [\Psi_T \vec{w}, \Psi_T u] \Vert_{1, T}^{-1} \langle \tilde{\mathcal{P}}_h(\director_h, \phi_h), (\Psi_T \vec{w}, \Psi_T u) \rangle \nonumber \\
& \leq \sup_{\substack{[\vec{v}_h, \psi_h] \in \tilde{Y}_{h \vert T} \\ \Vert [\vec{v}_h, \psi_h] \Vert_Y = 1}} \langle \tilde{\mathcal{P}}_h(\director_h, \phi_h), (\vec{v}_h, \psi_h) \rangle. \label{firstLowerBoundInequality}
\end{align}
Next, by applying \eqref{cutoff2} from Lemma \ref{cutoffinequalities} and observing that the integrals and norms are taken over $E$ where $P$ does not modify the values of either $\sigma$ or $\beta$,
\begin{align*}
&C_2 \bar{C}_6^{-1} C_7^{-1} h_{E}^{1/2} \Vert [\hat{\vec{p}}, \hat{q}] \Vert_{0, E} \nonumber \\
& \hspace{0.75in} \leq \sup_{[\sigma, \beta] \in \Pi_{k, l, E}} \frac{\bar{C}_6^{-1} h_E}{C_7 h_E^{1/2} \Vert [P \sigma, P \beta] \Vert_{0, E}} \int_E (\hat{\vec{p}}, \hat{q}) \cdot (\Psi_E P \sigma, \Psi_E P \beta)  \diff{S}.
\end{align*}
Now note that $\Psi_E P \sigma$ is supported on $\omega_E$ and that the norm in the denominator is taken over $E$. This implies that
\begin{align}
&C_2 \bar{C}_6^{-1} C_7^{-1} h_{E}^{1/2} \Vert [\hat{\vec{p}}, \hat{q}] \Vert_{0, E} \nonumber \\
&\leq \sup_{[\sigma, \beta] \in \Pi_{k, l, E}} \frac{\bar{C}_6^{-1} h_E}{C_7 h_E^{1/2} \Vert [\sigma, \beta] \Vert_{0, E}} \bigg ( \langle \tilde{\mathcal{P}}_h(\director_h, \phi_h), (\Psi_E P \sigma, \Psi_E P \beta) \rangle \nonumber \\
&\hspace{2.5in} - \int_{\omega_E} (\vec{p}, q) \cdot (\Psi_E P \sigma, \Psi_E P \beta) \diff{V} \bigg)\nonumber \\
& \leq \sup_{[\sigma, \beta] \in \Pi_{k, l, E}} \frac{\bar{C}_6^{-1} h_E}{\Vert [\Psi_E P \sigma, \Psi_E P \beta] \Vert_{0, \omega_E}} \bigg ( \langle \tilde{\mathcal{P}}_h(\director_h, \phi_h), (\Psi_E P \sigma, \Psi_E P \beta) \rangle \nonumber \\
&\hspace{2.5in} - \int_{\omega_E} (\vec{p}, q) \cdot (\Psi_E P \sigma, \Psi_E P \beta) \diff{V} \bigg) \label{C7Inequality},
\end{align}
where \eqref{C7Inequality} is given by \eqref{cutoff5} of Lemma \ref{cutoffinequalities} with $C_7$ properly modified to incorporate each element of $\omega_E$. Distributing the fraction and applying \eqref{expansion2} of Corollary \ref{cutoffexpansion}, to the first component and the Cauchy-Schwarz inequality to the second yields
\begin{align}
&C_2 \bar{C}_6^{-1} C_7^{-1} h_{E}^{1/2} \Vert [\hat{\vec{p}}, \hat{q}] \Vert_{0, E} \nonumber \\
& \leq \sup_{[\sigma, \beta] \in \Pi_{k, l, E}} \Vert [\Psi_E P \sigma, \Psi_E P \beta] \Vert_{1, \omega_E}^{-1} \langle \tilde{\mathcal{P}}_h(\director_h, \phi_h), (\Psi_E P \sigma, \Psi_E P \beta) \rangle \nonumber \\
& \hspace{1.5in} + \bar{C}_6^{-1} h_E \sum_{T \in \omega_E} \Vert [\hat{\vec{p}}, \hat{q}] \Vert_{0, T} \nonumber \\
& \leq \sup_{\substack{[\vec{v}_h, \psi_h] \in \tilde{Y}_{h \vert \omega_E} \\ \Vert [\vec{v}_h, \psi_h] \Vert_{Y} =1}} \Vert [\vec{v}_h, \psi_h] \Vert_{1, \omega_E}^{-1} \langle \tilde{\mathcal{P}}_h(\director_h, \phi_h), (\vec{v}_h, \psi_h) \rangle \nonumber \\
& \hspace{1.5in} + C_d \sup_{\substack{[\vec{v}_h, \psi_h] \in \tilde{Y}_{h \vert \omega_E} \\ \Vert [\vec{v}_h, \psi_h] \Vert_{Y} =1}} \langle \tilde{\mathcal{P}}_h(\director_h, \phi_h), (\vec{v}_h, \psi_h) \rangle, \label{RelateBackToPreviousLowerBound}
\end{align}
where the final inequality in \eqref{RelateBackToPreviousLowerBound} is given by expanding the space over which the supremum is taken in the first summand and using the inequality in \eqref{firstLowerBoundInequality}, with $C_d$ relating the constants $\bar{C}_6^{-1} h_E$ and $C_1 \bar{C}_4^{-1} h_T$ and taking care of the summation over $\omega_E$. Note that the supremums only increase when taken over $\omega_T$. Gathering the bounds in \eqref{firstLowerBoundInequality} and \eqref{RelateBackToPreviousLowerBound} and applying the inequality
\begin{align}
\left ( \sum_i a_i \right )^{1/2} \leq \sum_i a_i^{1/2}, \label{positiveAiLowerBound}
\end{align}
for $a_i > 0$, implies that
\begin{align}
\bar{C} \Theta_T \leq \sup_{\substack{[\vec{v}_h, \psi_h] \in \tilde{Y}_{h \vert \omega_T} \\ \Vert [\vec{v}_h, \psi_h] \Vert_{Y} =1}} \langle \tilde{\mathcal{P}}_h(\director_h, \phi_h), (\vec{v}_h, \psi_h) \rangle. \label{localLowerBound}
\end{align}
Finally, summing over $T \in \triangulation$ and applying \eqref{positiveAiLowerBound} again yields
\begin{align*}
C \left ( \sum_{T \in \triangulation} \Theta_T^2 \right)^{1/2} \leq \Vert \tilde{\mathcal{P}}_h(\director_h, \phi_h) \Vert_{\tilde{Y}_h^*}. \qquad \endproof
\end{align*}

These results enable the statement and proof of the main result of this section establishing reliability and local efficiency of the proposed a posteriori error estimator. 
\begin{theorem}
Say that $(\director_*, \phi_*)$ is a solution to Equation \eqref{PenaltyFOOC} satifying the assumptions of Proposition \ref{nonlinearErrorEstimation}. Let $(\director_h, \phi_h)$ be a discrete solution to Equation \eqref{discretePenaltyFOOC} such that $\Vert \mathcal{P}_h(\director_h, \phi_h) \Vert_{Y_h^*}  = 0$ and $(\director_h, \phi_h) \in B((\director_*, \phi_*), R)$. Then there exist $C_r, C_e > 0$, independent of $h$, such that
\begin{align}
&\Vert (\director_*, \phi_*) - (\director_h, \phi_h) \Vert_1 \leq C_r \left ( \sum_{T \in \triangulation} \Theta_T^2 \right)^{1/2}, \label{reliability} \\
&\Theta_T \leq C_e \Vert (\director_*, \phi_*) - (\director_h, \phi_h) \Vert_{1, \omega_T} \label{efficiency}.
\end{align}
\end{theorem}
\begin{proof}
Combining Lemmas \ref{restrictionUpperBound} and \ref{auxilliarySpaceLowerBound} implies the bound
\begin{align*}
\Vert (\text{Id}_Y - R_h)^* \tilde{\mathcal{P}}_h(\director_h, \phi_h) \Vert_{Y^*} &\leq C_0 \left ( \sum_{T \in \triangulation} \Theta_T^2 \right)^{1/2}  \leq C_1 \Vert \tilde{\mathcal{P}}_h (\director_h, \phi_h) \Vert_{\tilde{Y}_h^*}
\end{align*}
for $C_0, C_1 > 0$. Thus, the conditions of Proposition \ref{auxiliarySpaceInequality} are fulfilled. Therefore, with the results in Equations \eqref{zeroComponent1} and \eqref{zeroComponent2} and Lemma \ref{auxiliarySpaceUpperBound},
\begin{align*}
\Vert \mathcal{P}(\director_h, \phi_h) \Vert_{Y^*} \leq C_2 \Vert \tilde{\mathcal{P}}_h (\director_h, \phi_h) \Vert_{\tilde{Y}_h^*} \leq C_3 \left ( \sum_{T \in \triangulation} \Theta_T^2 \right)^{1/2}.
\end{align*}
It is straightforward to show that, for the defined Sobolev spaces and $D \mathcal{P}(\director_*, \phi_*) \in \text{Isom}(X_0, Y)$, Proposition \ref{nonlinearErrorEstimation} still holds. The upper bound from Proposition \ref{nonlinearErrorEstimation} then implies that
\begin{align*}
\Vert (\director_*, \phi_*) - (\director_h, \phi_h) \Vert_1 &\leq 2 \Vert D \mathcal{P}(\director_*, \phi_*)^{-1} \Vert_{\mathcal{L}(Y^*, X_0)} \Vert \mathcal{P}(\director_h, \phi_h) \Vert_{Y^*} \\
& \leq 2 C_3 \Vert D \mathcal{P}(\director_*, \phi_*)^{-1} \Vert_{\mathcal{L}(Y^*, X_0)} \left ( \sum_{T \in \triangulation} \Theta_T^2 \right)^{1/2}.
\end{align*}
Setting $C_r = 2C_3 \Vert D \mathcal{P}(\director_*, \phi_*)^{-1} \Vert_{\mathcal{L}(Y^*, X_0)}$ proves the inequality in \eqref{reliability}.

As noted in \cite[Remark 2.2]{Verfurth1}, the lower bound of Proposition \ref{nonlinearErrorEstimation} remains valid when restricted to appropriate norms over the open subset $\omega_T \subset \Omega$. Together with Inequality \eqref{localLowerBound}, this implies that
\begin{align*}
\Theta_T \leq C_4 \sup_{\substack{[\vec{v}_h, \psi_h] \in \tilde{Y}_{h \vert \omega_T} \\ \Vert [\vec{v}_h, \psi_h] \Vert_{Y} =1}} \langle \tilde{\mathcal{P}}_h(\director_h, \phi_h), (\vec{v}_h, \psi_h) \rangle &\leq C_4 \Vert \mathcal{P}(\director_h, \phi_h) \Vert_{Y^*_{\omega_T}} \\
& \leq \frac{1}{2} C_4 C_5 \Vert (\director_*, \phi_*) - (\director_h, \phi_h) \Vert_{1, \omega_T},
\end{align*} 
where $C_5$ is given by the value of the restriction of the norm $\Vert D \mathcal{P}(\director_*, \phi_*) \Vert_{\mathcal{L}(X_0, Y^*)}^{-1}$ from Proposition \ref{nonlinearErrorEstimation} to $X_{\omega_T} \subset X_0$ and $Y^*_{\omega_T} \subset Y^*$, the subspaces of $X_0$ and $Y^*$ limited to functions supported on $\omega_T$. Taking $C_e = \frac{1}{2} C_4 C_5$ proves \eqref{efficiency}.
\end{proof}

The results of Lemma \ref{cutoffinequalities} are equally applicable to meshes composed of quadrilateral or simplicial elements, as noted in \cite[Remark 3.5]{Verfurth2}. Thus, the results of this section extend to either type of mesh, satisfying equivalent conditions.

Using a similar approach for the Lagrange multiplier formulation yields a related operator associated with the first-order optimality conditions that includes terms associated with the Lagrange multiplier but excludes the $\zeta$ component. Addressing the Lagrange multiplier terms in the same manner as \cite{Verfurth2} for the pressure-related parts of the estimator corresponding to the stationary, incompressible Navier-Stokes equations produces a related, element-wise estimator,
\begin{align*}
\Theta_T &= \Bigg \{ h_T^2 \left( \Vert \vec{p}_0 \Vert_{0, T}^2 + \Vert q \Vert_{0, T}^2 \right) + \Vert  \director_h \cdot \director_h - 1\Vert^2_{0, T} \\
& \hspace{1in} + \sum_{E \in \mathcal{E}(T) \cap \mathcal{E}_{h, \Omega}} h_E \left ( \Vert \hat{\vec{p}} \Vert_{0, E}^2 + \Vert \hat{q} \Vert_{0, E}^2 \right) \Bigg \}^{1/2}, 
\end{align*}
for $T \in \triangulation$ where $\vec{p}_0 = \vec{p} + \lambda_h \director_h$, with $\zeta = 0$ in $\vec{p}$. 

As discussed in \cite{Emerson6}, there are unique theoretical challenges in extending the reliability and efficiency theory established above to the Lagrange multiplier system. Specifically, for continuum solution triplets $(\director_*, \phi_*, \lambda_*)$ satisfying the unit-length constraint, $\lambda$ may be freely perturbed and the triplet remains a solution. While a number of the theoretical results above are extendable to the Lagrange multiplier estimator using similar techniques to those of \cite{Verfurth2} for the Navier-Stokes equations, the propositions of Section \ref{preliminaries} require special consideration in order to properly address the non-local nature of $\lambda_*$. Though these modifications are the subject of future work, the numerical experiments of Section \ref{numerics} suggest that the Lagrange multiplier estimator performs well as part of AMR schemes. 

\section{Numerical Results} \label{numerics}

In this section, we apply the elastic error estimator of \cite{Emerson6} to problems with analytical solutions on both 2D and 3D domains. These solutions enable both numerical verification of the estimator's theoretical properties and evaluation of a collection of AMR marking schemes. In addition to the elastic estimator simulations, numerical experiments applying the coupled estimator are presented. The inclusion of both electric and flexoelectric coupling, paired with the Dirichlet boundary conditions, limits the availability of non-trivial analytical solutions for these systems. However, the numerical results suggest that the proposed coupled  estimator markedly increases simulation efficiency with comparable or superior performance across a number of metrics compared with uniform mesh refinement.

The algorithm to compute equilibrium solutions to the nonlinear variational systems discussed in Section \ref{model} employs nested iteration (NI) \cite{Starke1}, which begins on a specified coarsest grid. On each NI level, Newton iterations are performed, updating the solution approximation at each step. The stopping criterion for the iterations on each mesh is based on a tolerance of $10^{-4}$ for the approximation's conformance to the first-order optimality conditions in the standard $l_2$ norm. The resulting approximation is then interpolated to a finer grid, where Newton iterations continue. For each iteration an incomplete Newton correction is performed such that for a given iterate $\vec{u}_k$, the next Newton iterate is given by $\vec{u}_{k+1} = \vec{u}_k + \alpha \delta \vec{u}_h$, where $\alpha \leq 1$. While more sophisticated techniques exist \cite{Emerson3}, this simple approach effectively encourages strict adherence to the unit-length constraint manifold. The damping parameter, $\alpha$, begins at $0.2$ and increases by $0.2$ at each level of NI, to a maximum of $1.0$, as the finer features of the solution become increasingly resolved. For more details on the algorithm, see \cite{Emerson2}. The systems are discretized with $Q_2$ elements for components associated with $\director$ and $\phi$ and $Q_1$ elements for computations involving $\lambda$. Finally, the same non-dimensionalization parameters used in \cite{Emerson2} are applied.

On each level, AMR has three stages to produce the next finer mesh:
\begin{equation*}
\text{Estimate} \rightarrow \text{Mark} \rightarrow \text{Refine}.
\end{equation*}
For each $T \in \mathcal{T}_H$, the local estimator $\Theta_T$ is computed with respect to the coarse approximate solution $\vec{u}_H$. Elements of $\mathcal{T}_H$ are then marked for refinement through one of three strategies. Let $0 < \nu < 1$. In the simplest method, referred to throughout as ``fixed," a constant ratio, $\nu$, of coarse mesh cells, sorted by largest $\Theta_T$ value, are flagged. With $\Theta_M$ denoting the largest value of $\Theta_T$ on the coarse level, the second approach, introduced in \cite{Gui1, Gui2, Gui3} and termed ``bandwidth" here, marks a cell $T$ if $\Theta_T \geq (1 - \nu)\Theta_M$. The final method employed is Dorfler marking \cite{Dorfler1}, where $T \in \mathcal{T}_H$ is flagged if it is part of a minimal subset $\hat{\mathcal{T}}_H \subset \mathcal{T}_H$ such that $\sum_{T \in \hat{\mathcal{T}}_H} \Theta_T^2 \geq (1- \nu) \sum_{T \in \mathcal{T}_H} \Theta_T^2$. Any marked cells are refined through bisection to produce the next NI mesh. The grid management, discretizations, and adaptive refinement computations are implemented with the widely used \emph{deal.II} finite-element library \cite{BangerthHartmannKanschat2007}. 

The simulations here utilize meshes with rectangular elements. Therefore, adaptive refinement leads to the existence of hanging nodes. These nodes are dealt with in a standard way by constraining their values with the neighboring regular nodes to maintain continuity along the boundary. Additionally, a $1$-irregular mesh is maintained such that the number of hanging nodes on an edge is at most one. Finally, the theory developed in preceding sections assumes that the studied meshes satisfy the admissibility property. This assumption is valid for the coarsest mesh but, with the introduction of hanging nodes, no longer holds after the first AMR stage. While mesh discretizations employing simplices can maintain admissibility with adaptivity, grids composed purely of rectangular elements cannot. Thus, following the first level of refinement, the error estimator is applied heuristically.

In order to compare efficiency across different refinement techniques, an approximate work unit (WU) is calculated for each simulation. Assuming the presence of solvers that scale linearly with the number of non-zeros in the matrix, a WU is defined as the sum of the non-zeros in the discretized Hessian for each Newton step over the NI hierarchy divided by the number of non-zeros in a reference fine-grid Hessian. Below, the reference Hessian belongs to the finest level of uniform refinement, when available, or the finest mesh from the ``fixed" flagging strategy with largest $\nu$. Thus, a WU roughly approximates the work required by any full NI hierarchy in terms of assembling and solving a single linearization step for the reference Hessian when optimally scaling solvers are applied. While the linear systems here are solved with simple LU decomposition, the reported WUs provide a best-case scaling for comparing the work required between refinement strategies.

\subsection{2D Elastic System Results}

The simulations in this section consider a unit-square domain with $K_i = 1$, $i = 1,2,3$. In this equal Frank constant case, the minimization reduces to a kind of harmonic mapping problem with known analytical solutions. Specifically, we examine an example from the family of solutions derived in \cite{Hu2} of the form
\begin{align} \label{2dHarmonicSol}
\director_* = (\sin \theta, \cos \theta, 0), && \theta = -4.5 \log \left(\sqrt{(x-0.5)^2 + (y + 0.1)^2} \right).
\end{align}
For all simulations, NI begins on a $32 \times 32$ coarse grid and, where applicable, the penalty parameter is $10^8$. For uniform refinement, five refinement levels are used, while experiments applying AMR continue until the number of fine-mesh degrees of freedom (DOFs) is larger than the finest uniform grid.

\begin{figure}[h!]
\begin{minipage}[t]{0.33 \textwidth}
\vspace{0pt}
\begin{subfigure}{0.99 \textwidth}
\includegraphics[scale=0.215, left]{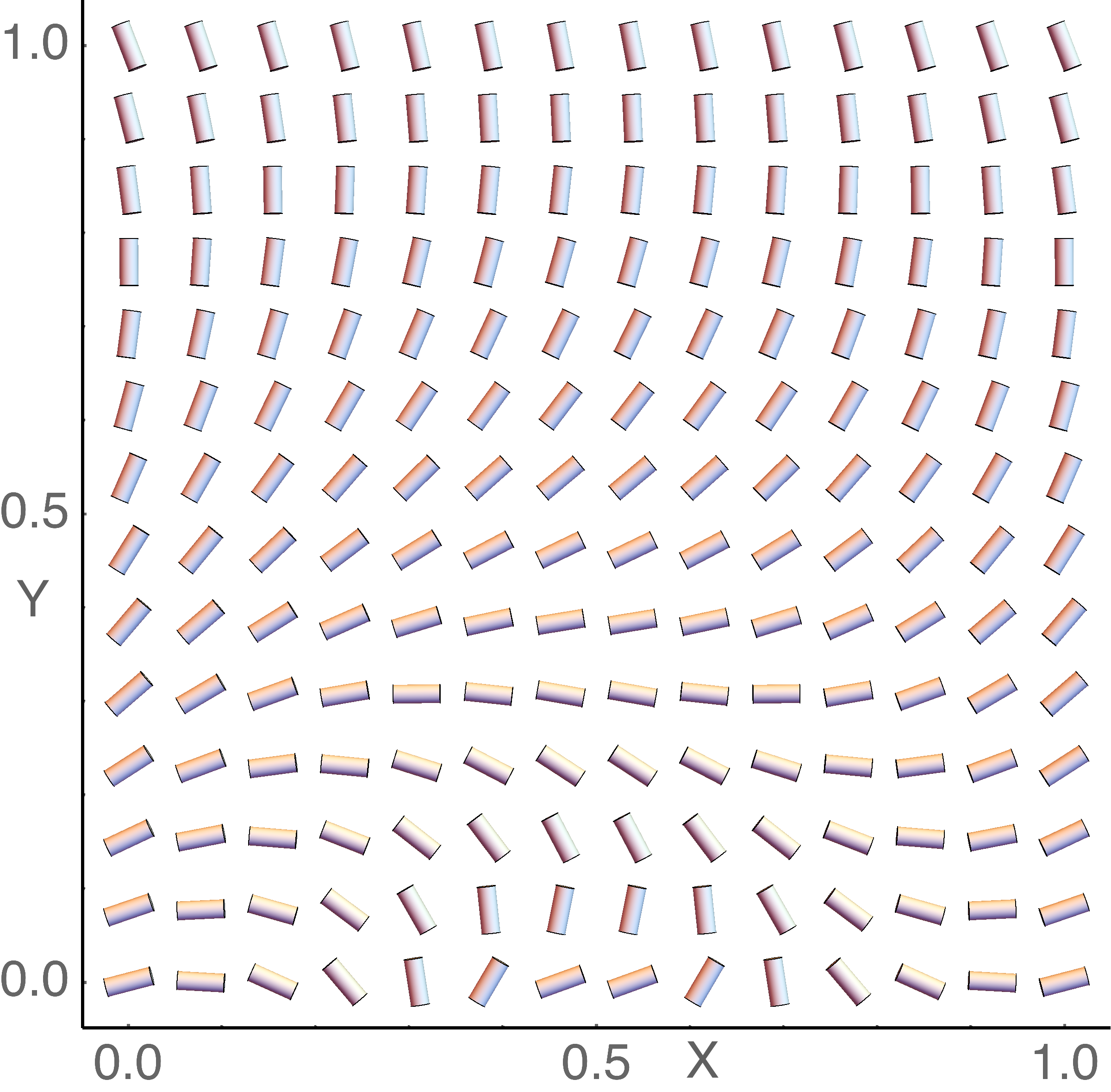}
  \caption{}
  \label{2DRefinementPlots:left}
\end{subfigure}
\end{minipage}
\begin{minipage}[t]{0.64 \textwidth}
\vspace{0pt}
\begin{subfigure}{0.99 \textwidth}
	\begin{subfigure}{0.53 \textwidth}
		  \includegraphics[scale=.073, left]{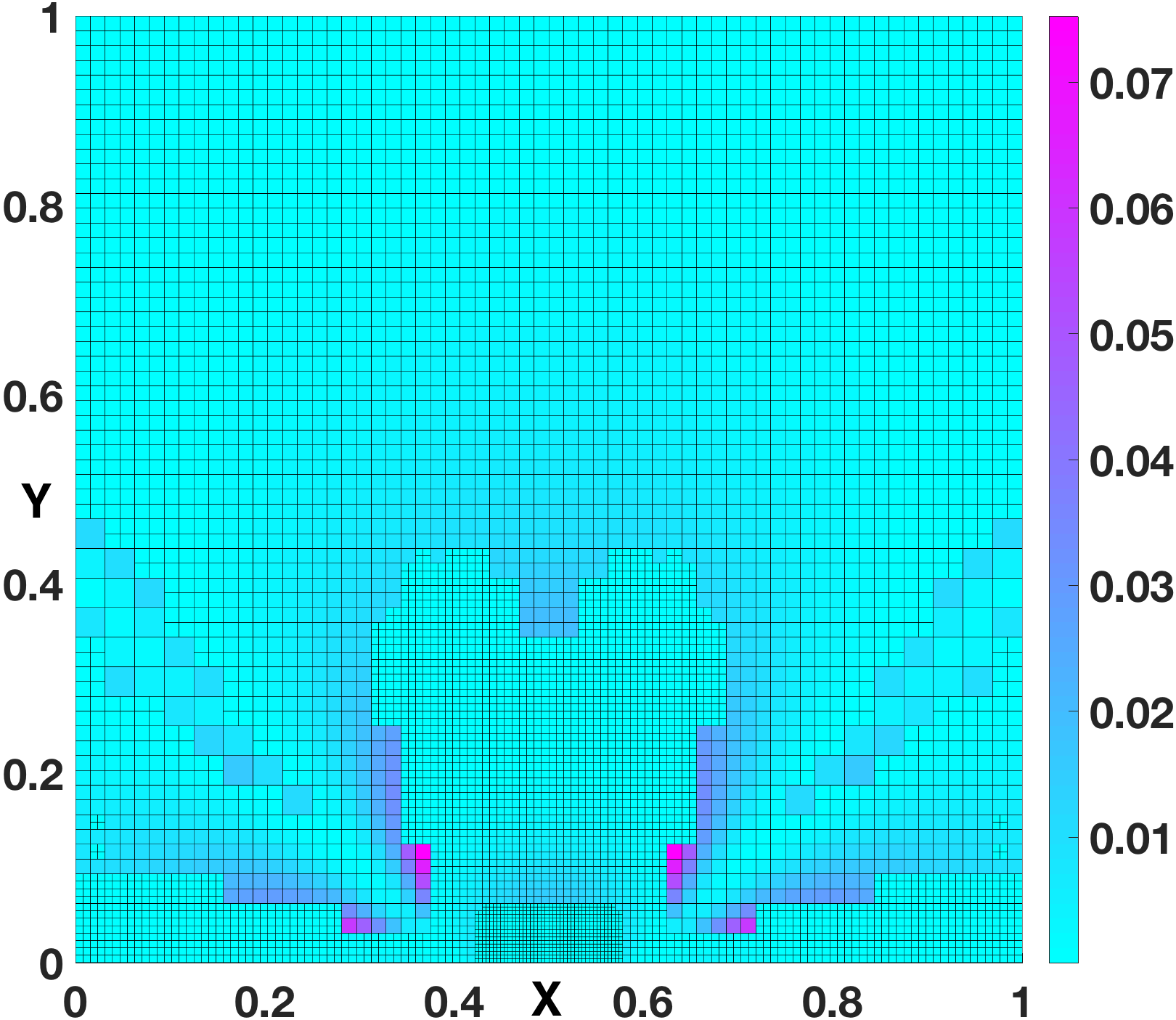}
	\end{subfigure}
	\begin{subfigure}{0.46 \textwidth}
		\includegraphics[scale=.073, left]{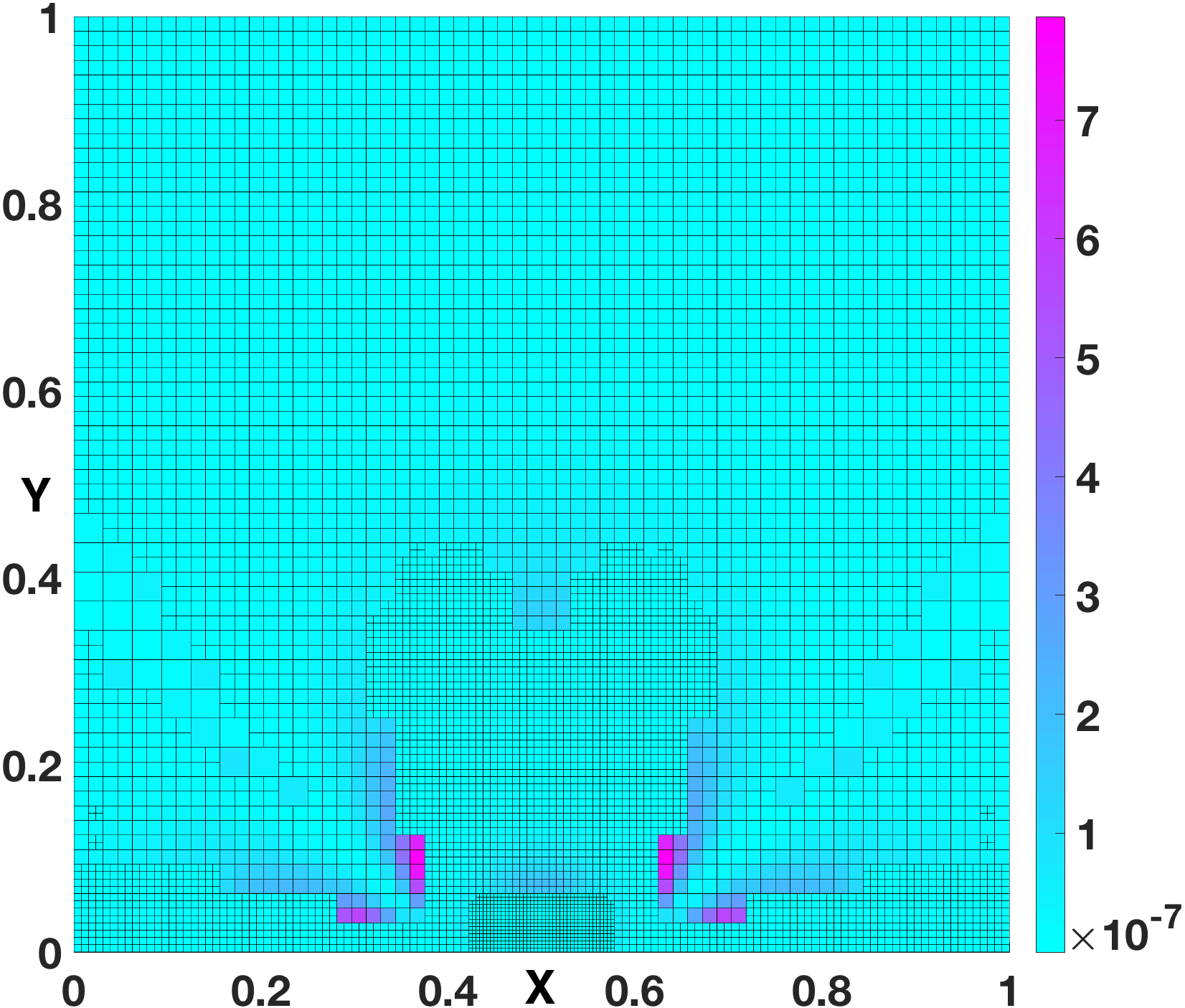}
	\end{subfigure}
  \caption{}
  	  \label{2DRefinementPlots:center1}
  \begin{subfigure}{0.99 \textwidth}
	\begin{subfigure}{0.53 \textwidth}
		  \includegraphics[scale=.073, left]{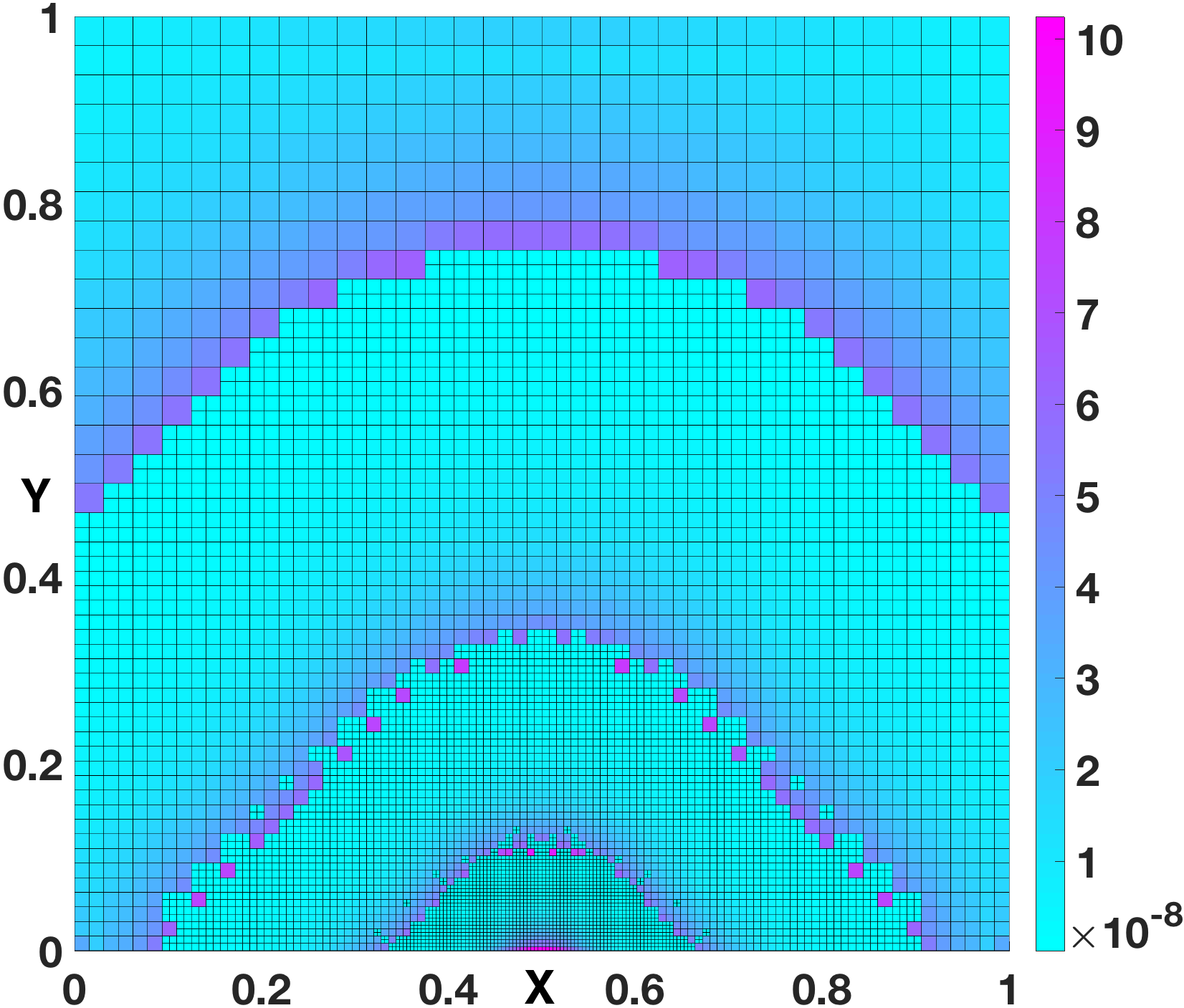}
	\end{subfigure}
	\begin{subfigure}{0.46 \textwidth}
		\includegraphics[scale=.073, left]{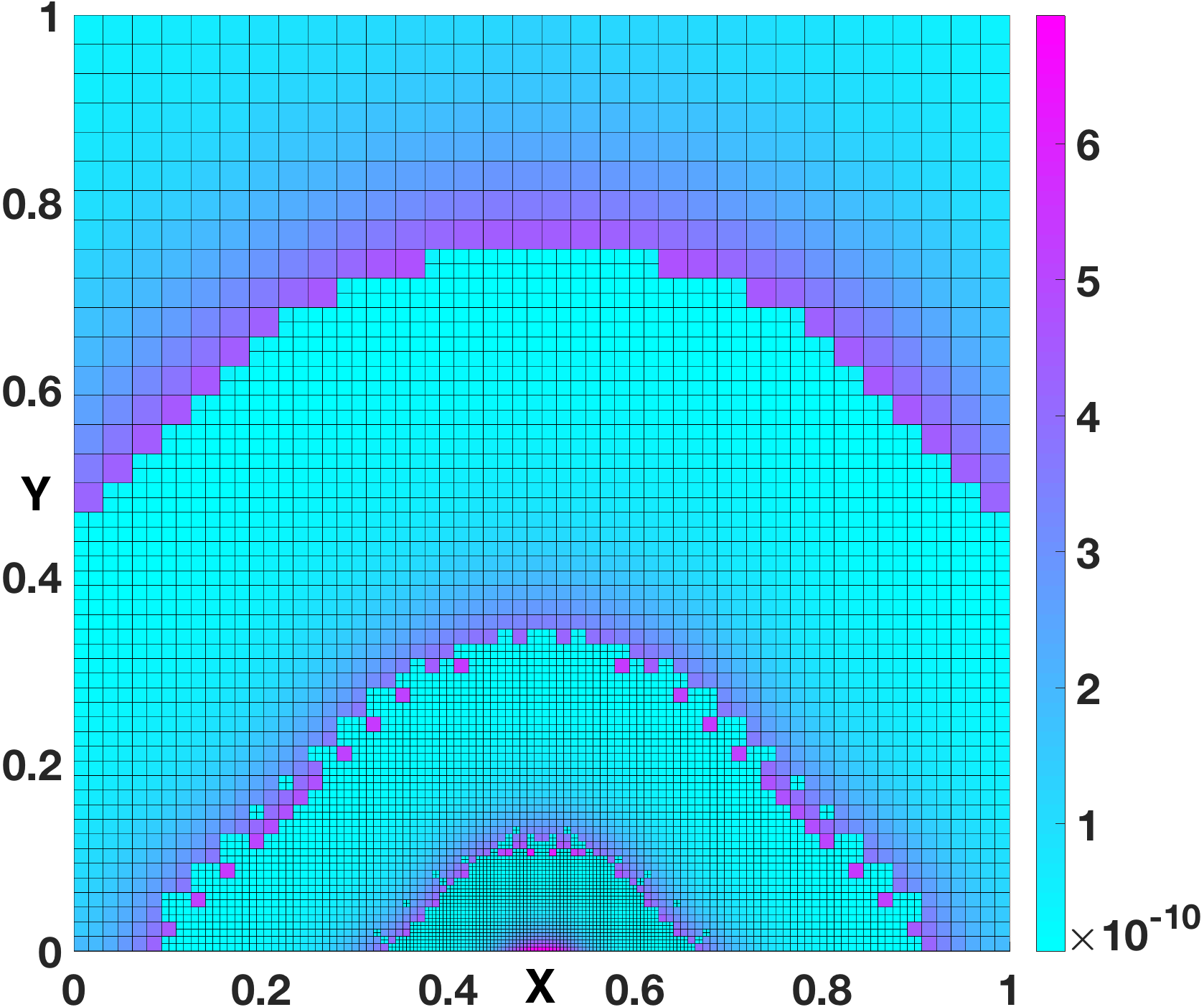}
	\end{subfigure}
  \caption{}
\end{subfigure}
	  \label{2DRefinementPlots:center2}
\end{subfigure}
\end{minipage}
\caption{\small{(\subref{2DRefinementPlots:left}) Solution on the finest adaptively refined mesh for the penalty method with bandwidth flagging (restricted for visualization). For the penalty method (\subref{2DRefinementPlots:center1}) and Lagrangian formulation (\subref{2DRefinementPlots:center2}),  a comparison of $\Theta_T^2$ (left) on each cell with $\Vert \director_* - \director_h \Vert_{1, T}$ (right)}}
\label{2DRefinementPlots}
\end{figure}
\vspace{-0.15cm}

As an example, the solution computed by the penalty method using the bandwidth flagging scheme with $\nu = 0.9$ is shown in Figure \ref{2DRefinementPlots}\subref{2DRefinementPlots:left}. The configuration is qualitatively indistinguishable from the analytical solution of \eqref{2dHarmonicSol} and matches the true free energy of $8.717$. Figures \ref{2DRefinementPlots}\subref{2DRefinementPlots:center1} and \subref{2DRefinementPlots:center2} exhibit a comparison of the local estimator, $\Theta_T^2$, on each cell to the analytical approximation error in the $H^1$-norm for the Lagrange multiplier and penalty methods, respectively, with bandwidth flagging after three AMR levels. In both cases, there is good agreement between areas of elevated estimator and error values. This correspondence is observed across each marking scheme and implies that the estimator is highly effective at identifying regions where additional refinement most effectively reduces approximation error. 

After establishing an accurate estimate of the local error for a computed solution, the method used to tag cells for refinement becomes an important component in generating near optimal discretizations. The graphs of Figure \ref{2DErrorAndRefinement} present the results of applying the three different flagging schemes, compared with uniform refinement for both constraint enforcement approaches. The $\nu$ values are set to the optimal value observed for each scheme. In each of the figures, it is clear that all marking approaches significantly outperform uniform refinement. Such behavior is observed even for non-optimal values of $\nu$. Figures \ref{2DErrorAndRefinement}\subref{2DErrorAndRefinement:left1} and \ref{2DErrorAndRefinement}\subref{2DErrorAndRefinement:left2} show reduction of the estimator and approximation error as a function of cells in the discretization for the penalty and Lagrange multiplier approaches, respectively. Note that the estimator remains an upper bound on the approximation error throughout the NI hierarchies, with both quantities showing similar reduction profiles as refinement progresses. The graphs indicate that, for early refinement levels, bandwidth and Dorfler flagging are more efficient than the fixed approach. However, with additional refinement the methods become comparable. A portion of this confluence is likely due to a combination of more uniformly distributed error and the $1$-irregularity mesh constraint forcing larger numbers of cells to be refined than tagged by either bandwidth and Dorfler, thereby slightly reducing their efficiency. Simplicial meshes that maintain regularity with refinement could show even better performance with these two flagging techniques. 

\begin{figure}[h!]
\begin{subfigure}{0.49 \textwidth}
\includegraphics[scale=0.232, center]{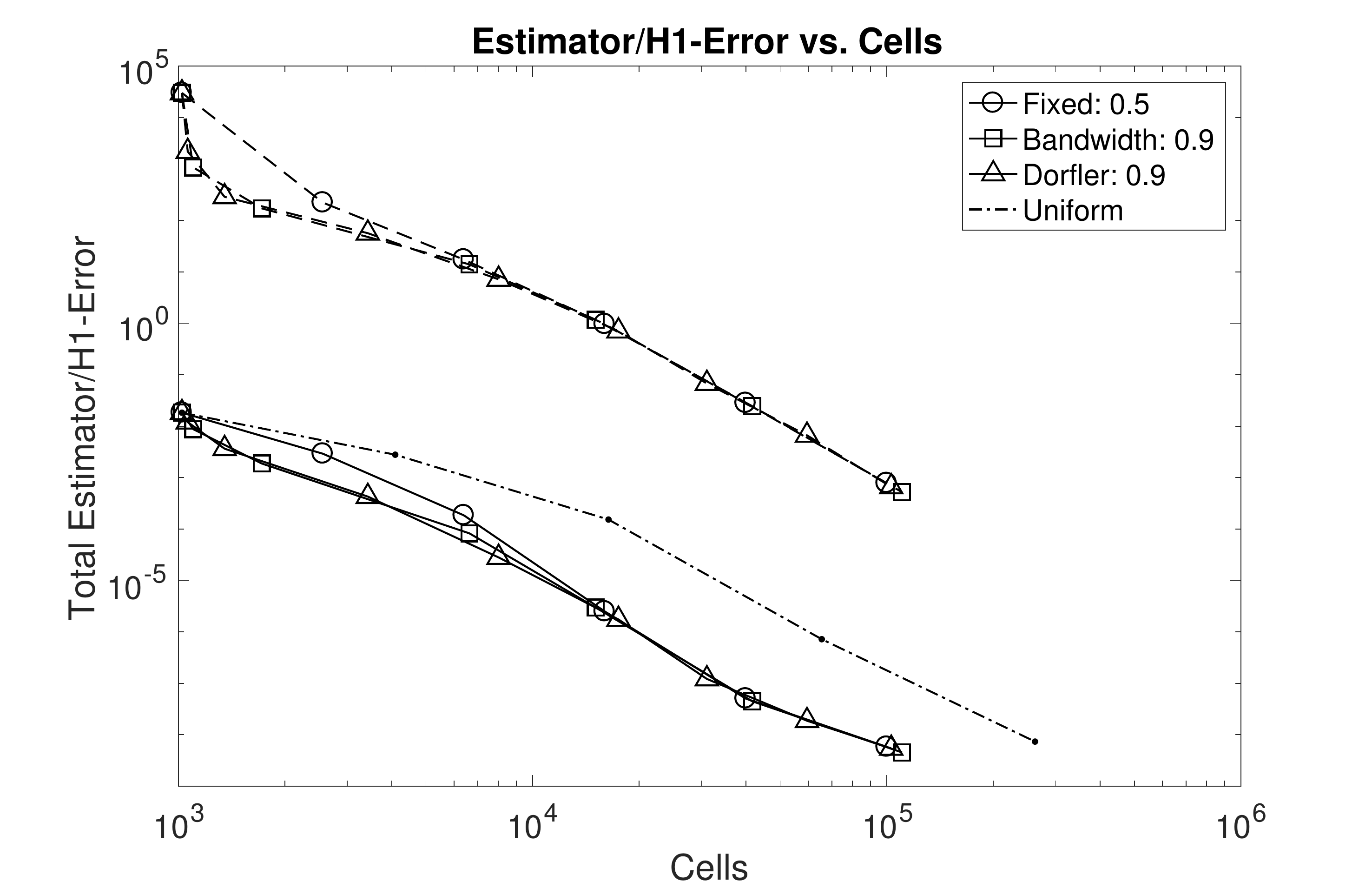}
  \caption{}
  \label{2DErrorAndRefinement:left1}
\end{subfigure}
\begin{subfigure}{0.49 \textwidth}
	\includegraphics[scale=0.232, center]{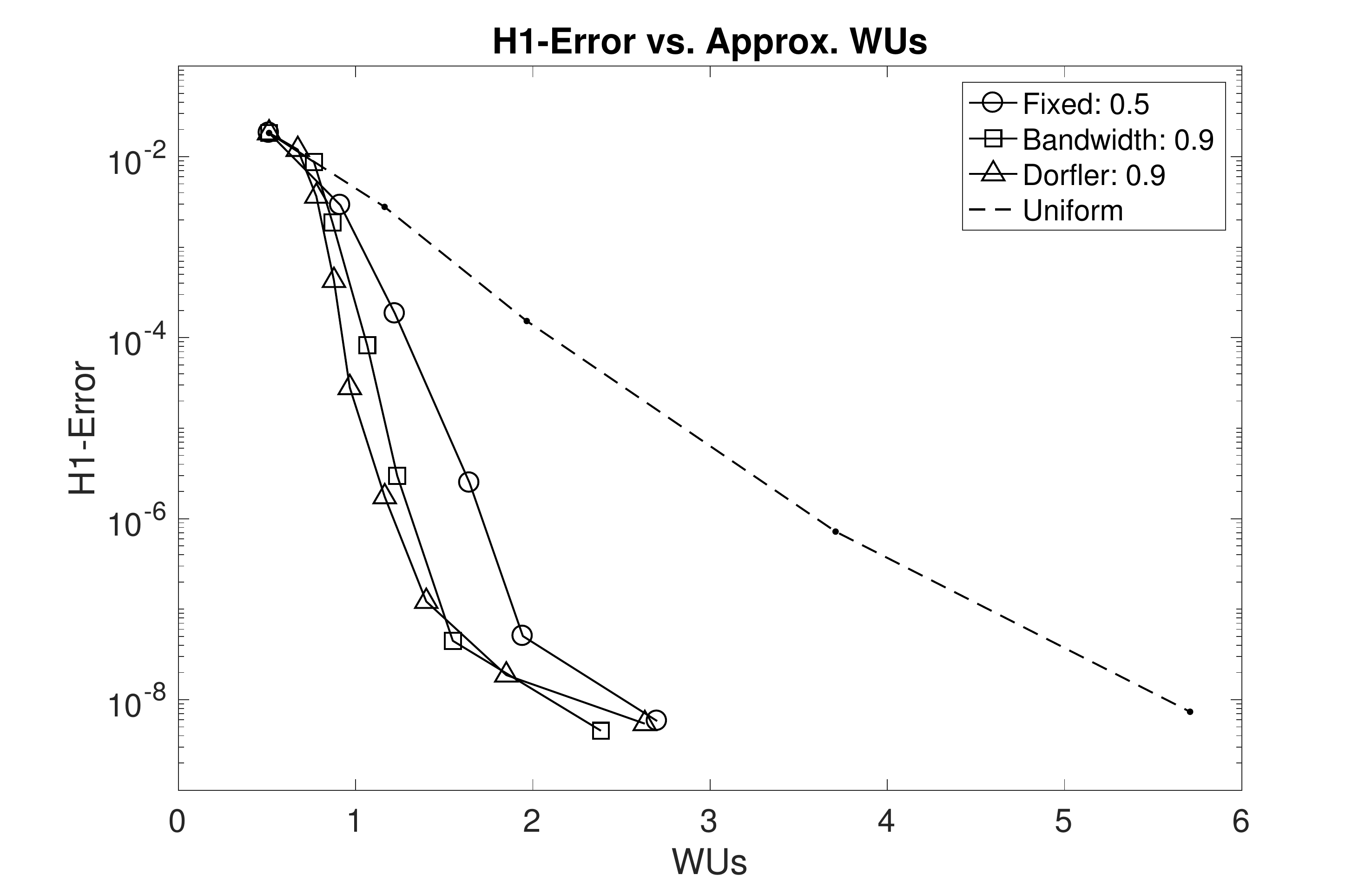}
  \caption{}
  \label{2DErrorAndRefinement:right1}
\end{subfigure}
\raggedleft
\begin{subfigure}{0.49 \textwidth}
\includegraphics[scale=0.232, center]{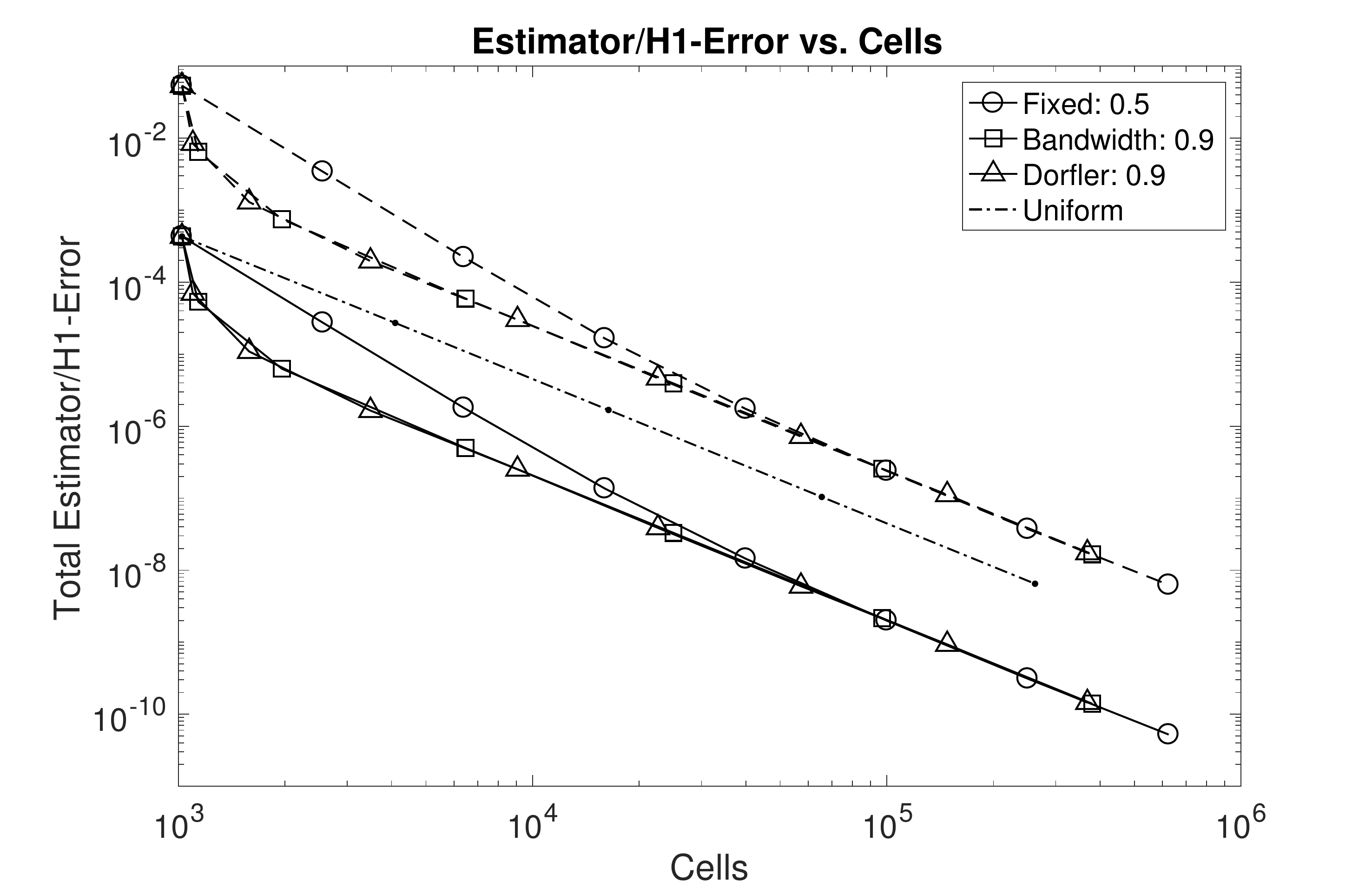}
  \caption{}
  \label{2DErrorAndRefinement:left2}
\end{subfigure}
\raggedright
\begin{subfigure}{0.49 \textwidth}
	\includegraphics[scale=0.232, center]{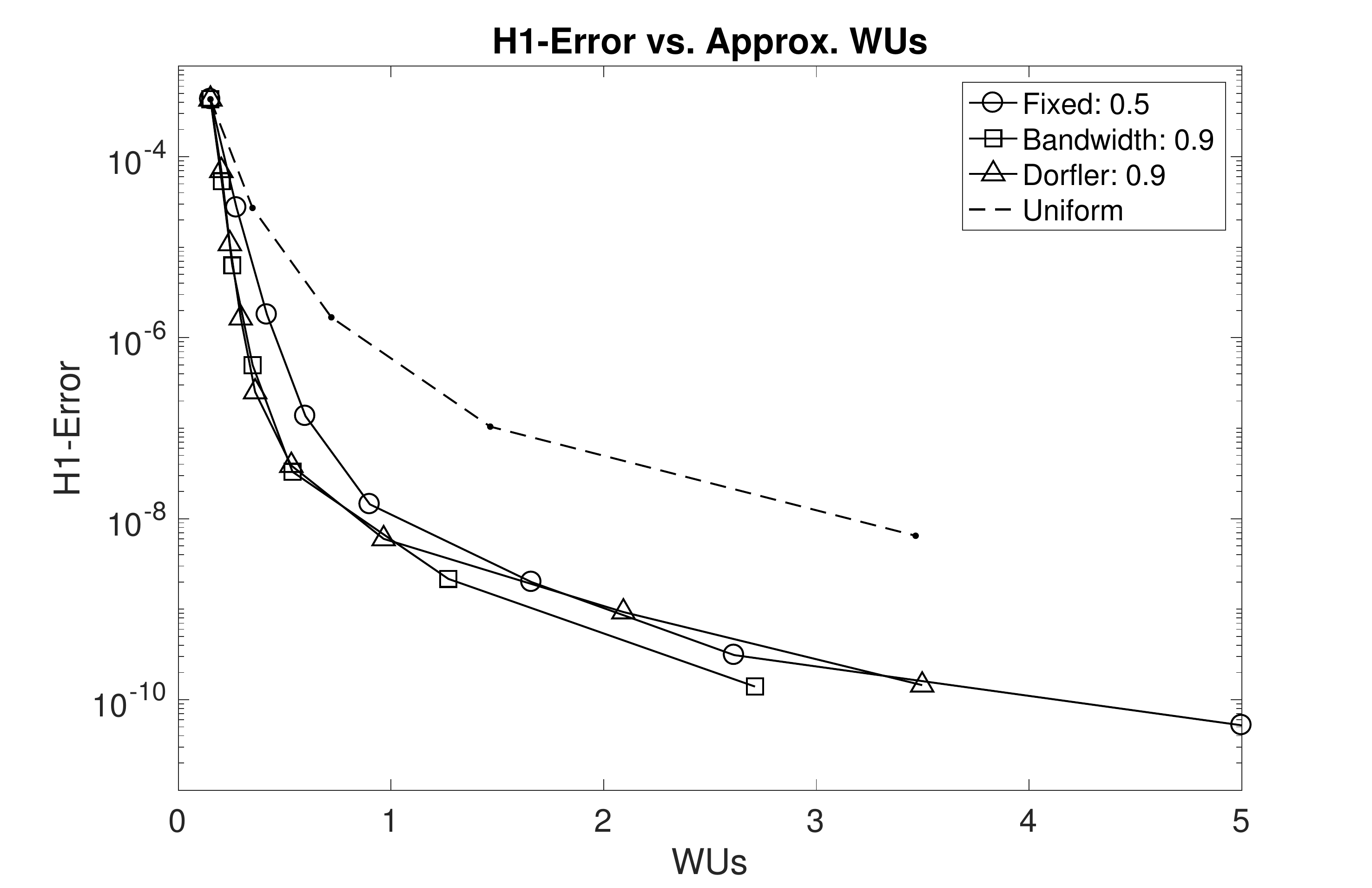}
  \caption{}
  \label{2DErrorAndRefinement:right2}
\end{subfigure}
\caption{\small{
(\subref{2DErrorAndRefinement:left1}, \subref{2DErrorAndRefinement:left2}) Reduction in the total estimator (dashed lines) and $H^1$ approximation error (solid lines) as a function of mesh elements at each level of refinement. (\subref{2DErrorAndRefinement:right1}, \subref{2DErrorAndRefinement:right2}) Reduction in $H^1$-error as a function of approximate WUs. The top row corresponds to the penalty method while the bottom is associated with the Lagrangian formulation.}}
\label{2DErrorAndRefinement}
\end{figure}
\vspace{-0.15cm}

The reduction of approximation error as a function of WUs is displayed in Figures \ref{2DErrorAndRefinement}\subref{2DErrorAndRefinement:right1} and \ref{2DErrorAndRefinement}\subref{2DErrorAndRefinement:right2}. For the Lagrange multiplier method, AMR achieves at least two orders of magnitude better error while consuming the same or fewer WUs compared to the uniformly refined meshes. In the penalty case, uniformly refined mesh required twice the WUs to reach an equivalent error. In either case, the bandwidth tagging approach performed somewhat better than the Dorfler scheme. Note that for coarser mesh with the penalty method, error and estimator reduction is slightly less uniform. The large penalty parameter for this problem strongly influences the error estimator in regions with heavier violations of the unit-length constraint. Therefore, the sharpness of the estimator is reduced, in practice, until the unit-length constraint is well satisfied. Finally, Figure \ref{2DErrorDistributions} displays illustrative examples of the difference in the distribution of estimator and error quantities on the uniform coarsest and adaptively refined finest meshes for the Lagrange multiplier method. Each point represents the fraction of total estimator or error value contained in the corresponding percentage of mesh cells, ordered by each cell's contribution to the quantity. In \cite{Gui1, Gui2, Gui3}, it is shown that near optimal discretizations are achieved, for one-dimensional problems, by equally distributing error across mesh elements. It is believed that this result extends to higher dimensions. Figure \ref{2DErrorDistributions}\subref{2DErrorDistributions:left1} suggests that for nearly uniform AMR, where $90\%$ of elements are refined at each level, very little progress towards equal distribution of the error or estimator values is achieved. In contrast, significant improvement in the distribution of both quantities is achieved with more targeted flagging techniques, as seen in Figure \ref{2DErrorDistributions}\subref{2DErrorDistributions:right1} exhibiting results with bandwidth marking.

\begin{figure}[h!]
\begin{subfigure}{0.49 \textwidth}
\includegraphics[scale=0.232, center]{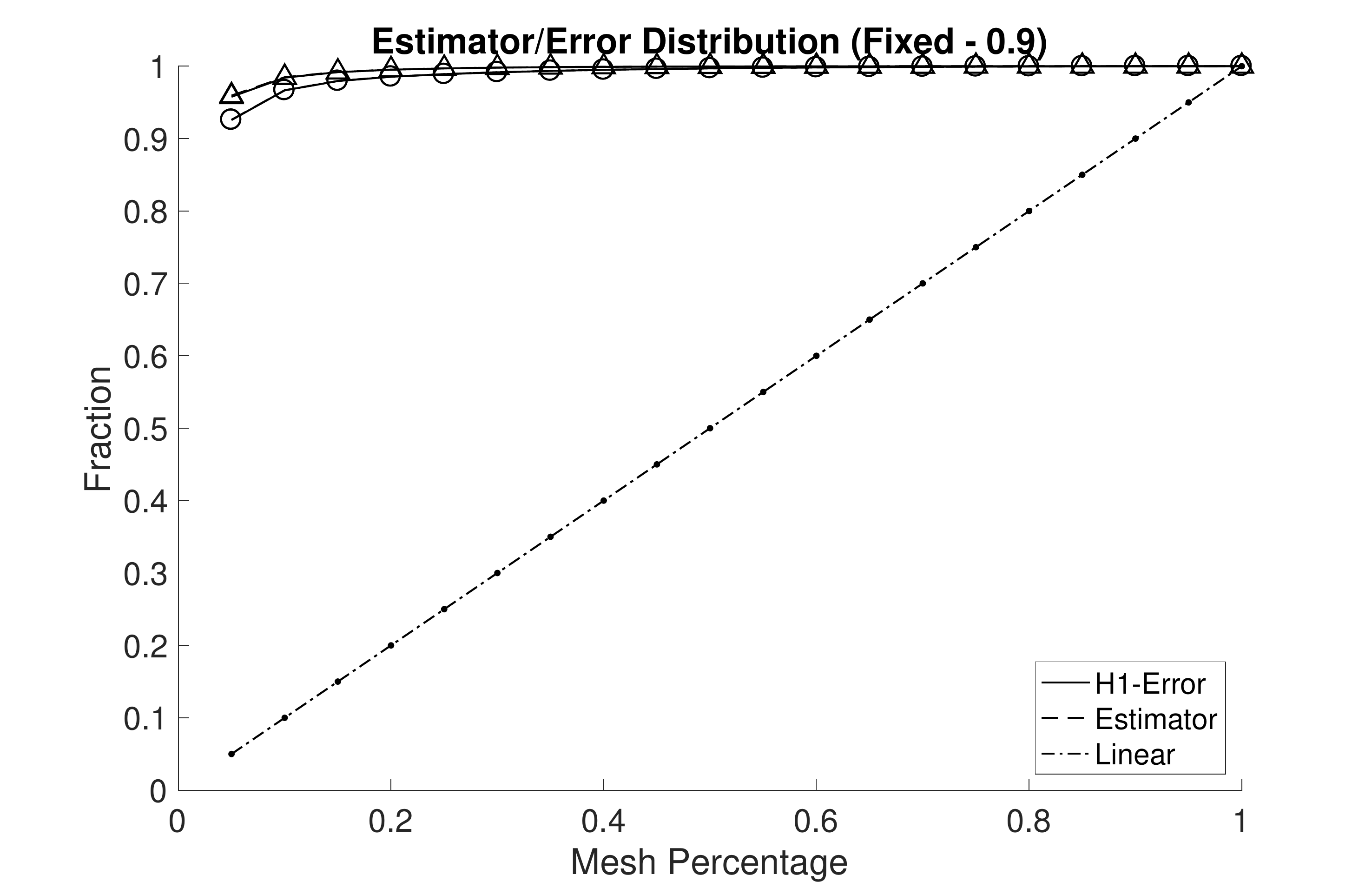}
  \caption{}
  \label{2DErrorDistributions:left1}
\end{subfigure}
\begin{subfigure}{0.49 \textwidth}
	\includegraphics[scale=0.232, center]{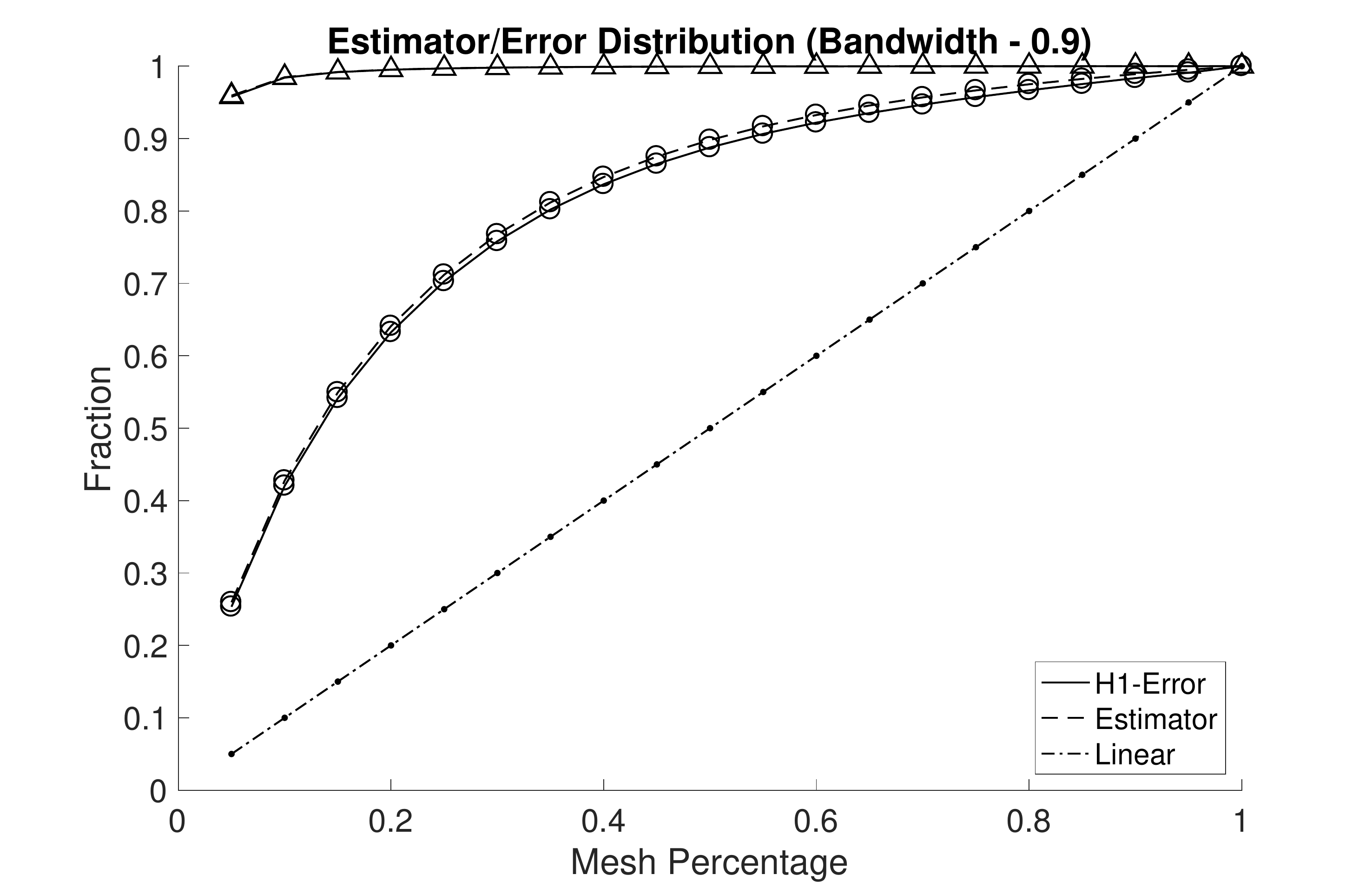}
  \caption{}
  \label{2DErrorDistributions:right1}
\end{subfigure}
\caption{\small{
The distribution of the total estimator value or $H^1$ approximation error across the cells of the coarsest and finest meshes after reaching the Newton stopping criterion for (\subref{2DErrorDistributions:left1}) the fixed marking scheme with $\nu = 0.9$ and (\subref{2DErrorDistributions:right1}) bandwidth AMR using $\nu = 0.9$. Markers $\bigtriangleup$ and $\Circle$ delineate intial and final mesh distributions, respectively.
}}
\label{2DErrorDistributions}
\end{figure}
\vspace{-0.3cm}

\subsection{3D Elastic System Results}

In this section, we examine the performance of AMR for a liquid crystal sample with equal Frank constants confined to a unit cube. Following the construction detailed in \cite{Alouges1, Cohen1}, we consider a subset of solutions of the form $\vec{u}_{\gamma_0}(\vec{x}) = \pi^{-1} \circ \gamma_0 \circ \pi \left( \frac{\vec{x}}{\vert \vec{x} \vert} \right)$. For this example,
\begin{align*}
\pi(x, y, z) &= \left(1 - z\right)^{-1} \left ( x, y \right), \\
\pi^{-1}(x, y) &= \left( 1+x^2+y^2 \right)^{-1} \left( 2x, 2y, x^2 + y^2 - 1 \right), \\
\gamma_0(x, y) &= \left (\frac{x}{x^2+ y^2} + x^2 - y^2,  2xy - \frac{y}{x^2 + y^2}\right).
\end{align*}
A simple shift $w(x, y, z) = (x + 0.2, y + 0.1, z)$ is employed to remove the singularity at the origin yielding an analytical equilibrium solution of the form $\director_* = \vec{u}_{\gamma_0} \circ w$.

For the experiments of this section, NI begins on an $8^3$ mesh and the penalty parameter is $\zeta = 10^6$, where applicable. Due to the rapid growth in problem size with uniform refinement, computations using uniform mesh are not reported. The configurations in Figures \ref{3DSolutionAndRefinement}\subref{3DSolutionAndRefinement:left1} and \ref{3DSolutionAndRefinement}\subref{3DSolutionAndRefinement:right1} are slices at $y = 0.2$ and $z = 0.1$, respectively, of the computed solution using the penalty method and Dorfler AMR with $\nu = 0.9$. On the finest mesh, the calculated free energy of $8.847$ matches that of the analytical solution. Furthermore, Figures \ref{3DSolutionAndRefinement}\subref{3DSolutionAndRefinement:left2} and \ref{3DSolutionAndRefinement}\subref{3DSolutionAndRefinement:right2} show the resulting meshes after four refinement stages for the penalty and Lagrange multiplier methods, respectively, with Dorfler marking. The meshes are overlaid on the corresponding coarse-grid $H^1$ approximation error after the Newton iteration tolerance is reached. Notably, the regions emphasized by the refinement process coincide with the areas of largest error.

\begin{figure}[h!]
\begin{subfigure}{0.49 \textwidth}
\includegraphics[scale=0.23, center]{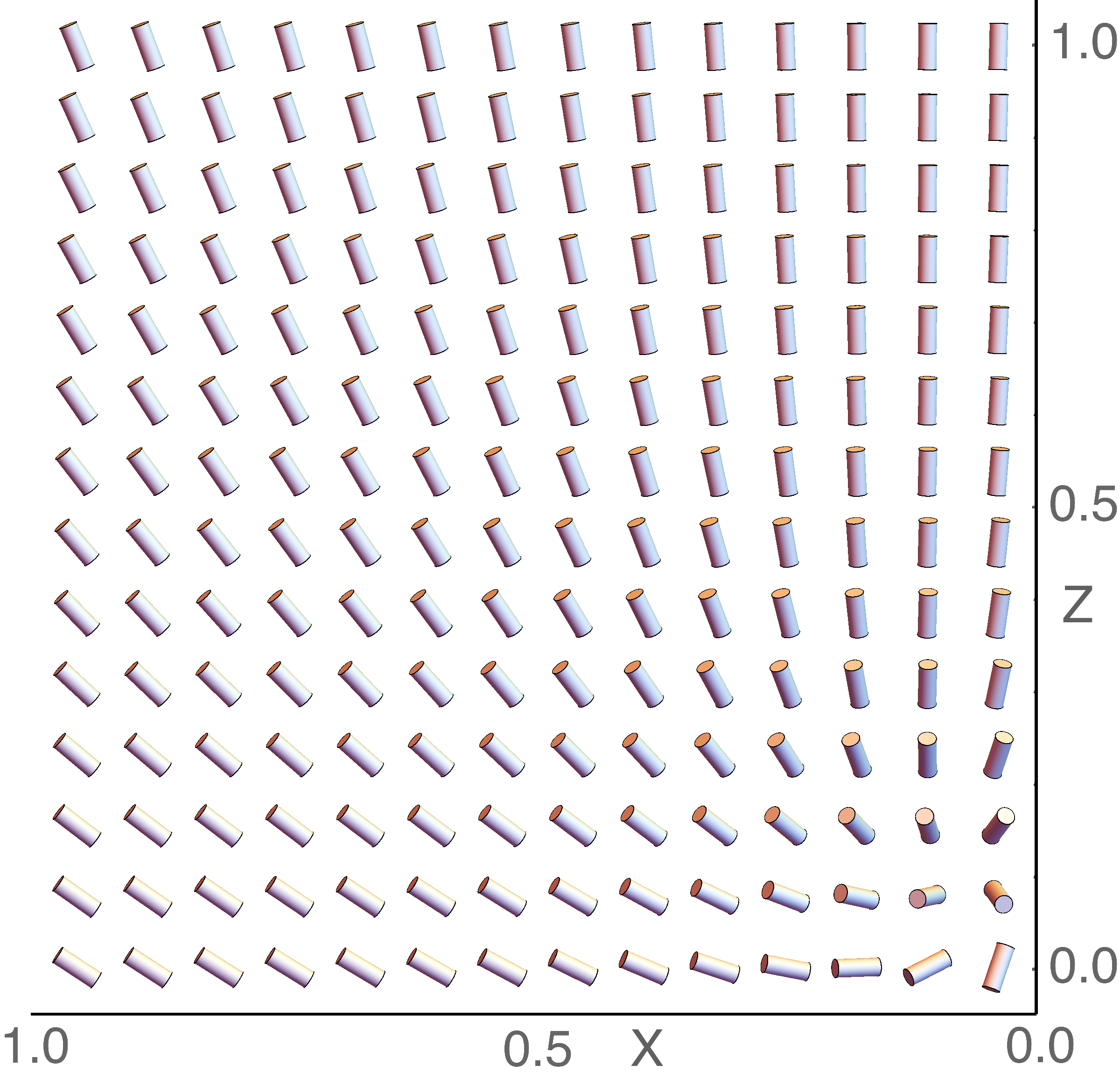}
  \caption{}
  \label{3DSolutionAndRefinement:left1}
\end{subfigure}
\begin{subfigure}{0.49 \textwidth}
\includegraphics[scale=0.23, center]{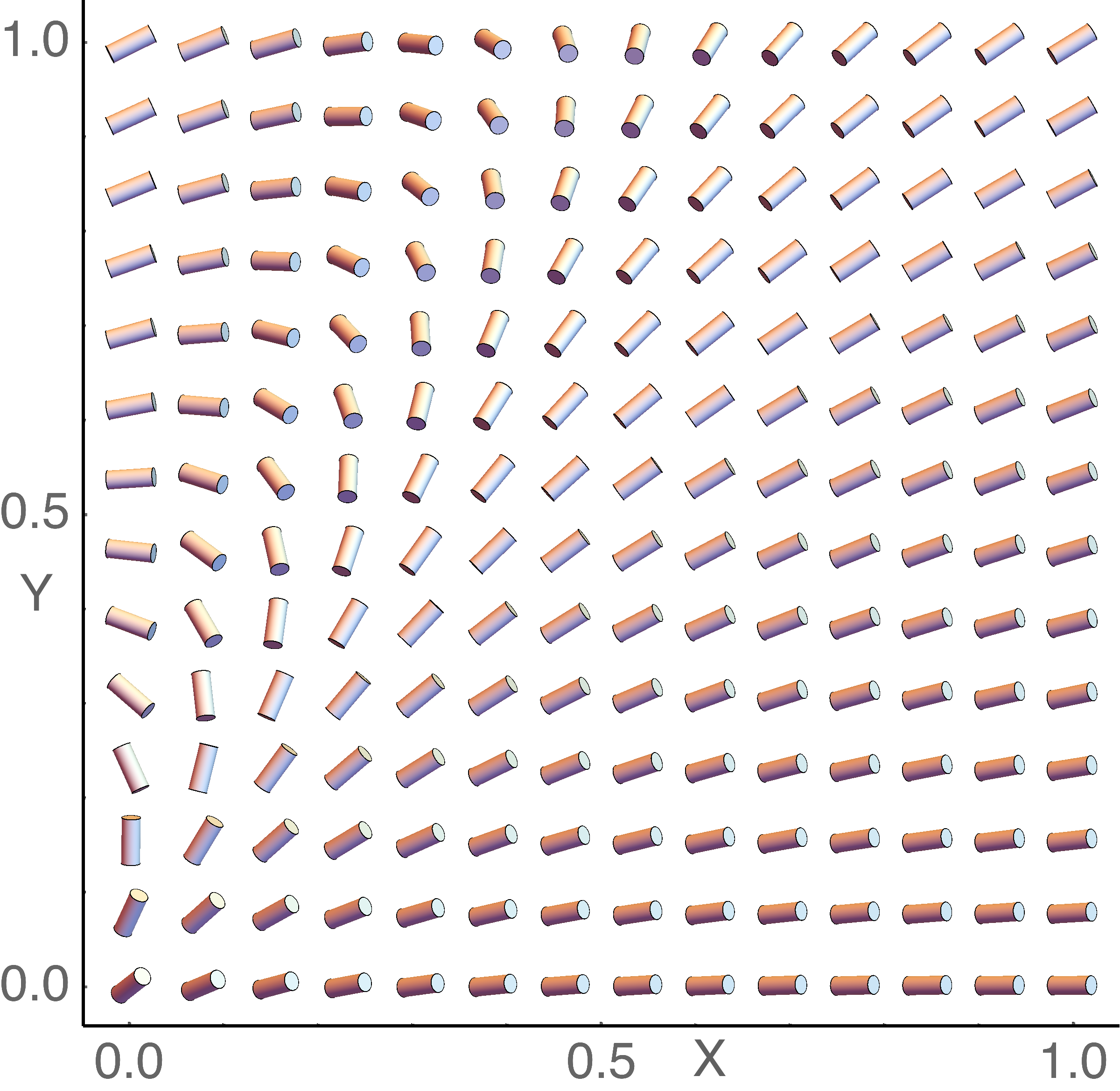}
  \caption{}
  \label{3DSolutionAndRefinement:right1}
\end{subfigure}
\begin{subfigure}{0.49 \textwidth}
	\includegraphics[scale=0.16, center]{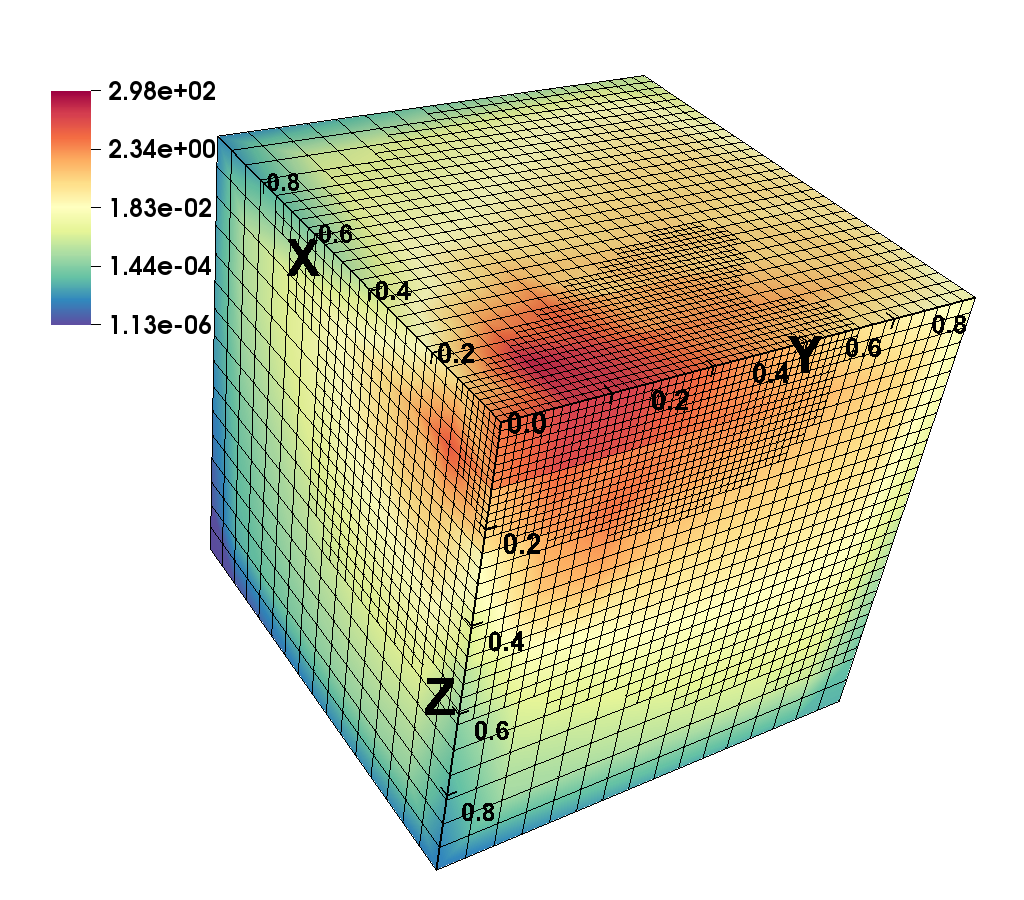}
  \caption{}
  \label{3DSolutionAndRefinement:left2}
\end{subfigure}
\begin{subfigure}{0.49\textwidth}
	\includegraphics[scale=0.16, center]{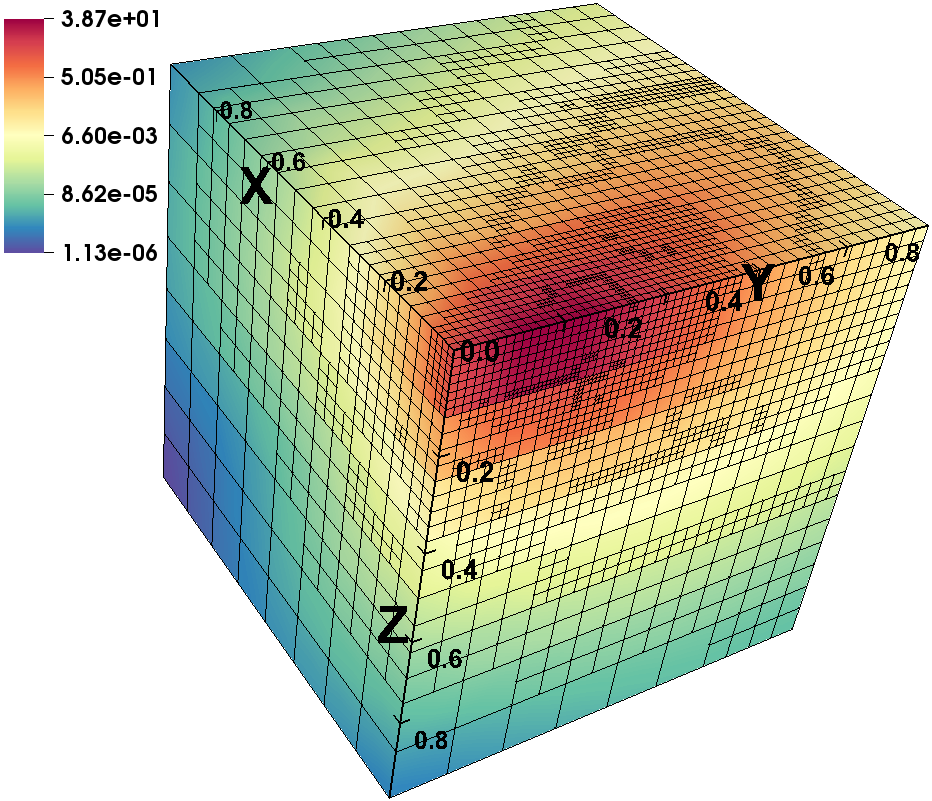}
  \caption{}
  \label{3DSolutionAndRefinement:right2}
\end{subfigure}
\caption{\small{
(\subref{3DSolutionAndRefinement:left1}, \subref{3DSolutionAndRefinement:right1}) Slices at $y = 0.2$ and $z = 0.1$, respectively, of the solution on the finest mesh (restricted for visualization) computed with the penalty method and Dorfler marking. (\subref{3DSolutionAndRefinement:left2}, \subref{3DSolutionAndRefinement:right2}) The resulting AMR patterns for the penalty and Lagrangian formulations after four levels of refinement overlaid on the coarse-grid $H^1$-error after Newton convergence.
}}
\label{3DSolutionAndRefinement}
\end{figure}

Paired with adaptive refinement guided by the proposed error estimators, each flagging approach is highly effective and efficient at reducing the error in the computed solution. The plots in Figures \ref{3DErrorAndRefinement}\subref{3DErrorAndRefinement:left1} and \ref{3DErrorAndRefinement}\subref{3DErrorAndRefinement:right1} display the reduction in overall $H^1$-error for each of the flagging schemes as a function of consumed WUs for the penalty and Lagrange multiplier methods. As in the previous section, optimal values of $\nu$ for each of the flagging methods are shown. In general, the bandwidth and Dorfler schemes appear to outperform the fixed approach, with a more pronounced improvement observable in the penalty method simulations. Finally, Dorfler AMR is slightly more efficient and has a more consistent optimal $\nu$ value of $0.9$ across experiments.

\begin{figure}[h!]
\begin{subfigure}{0.49 \textwidth}
	\includegraphics[scale=0.232, center]{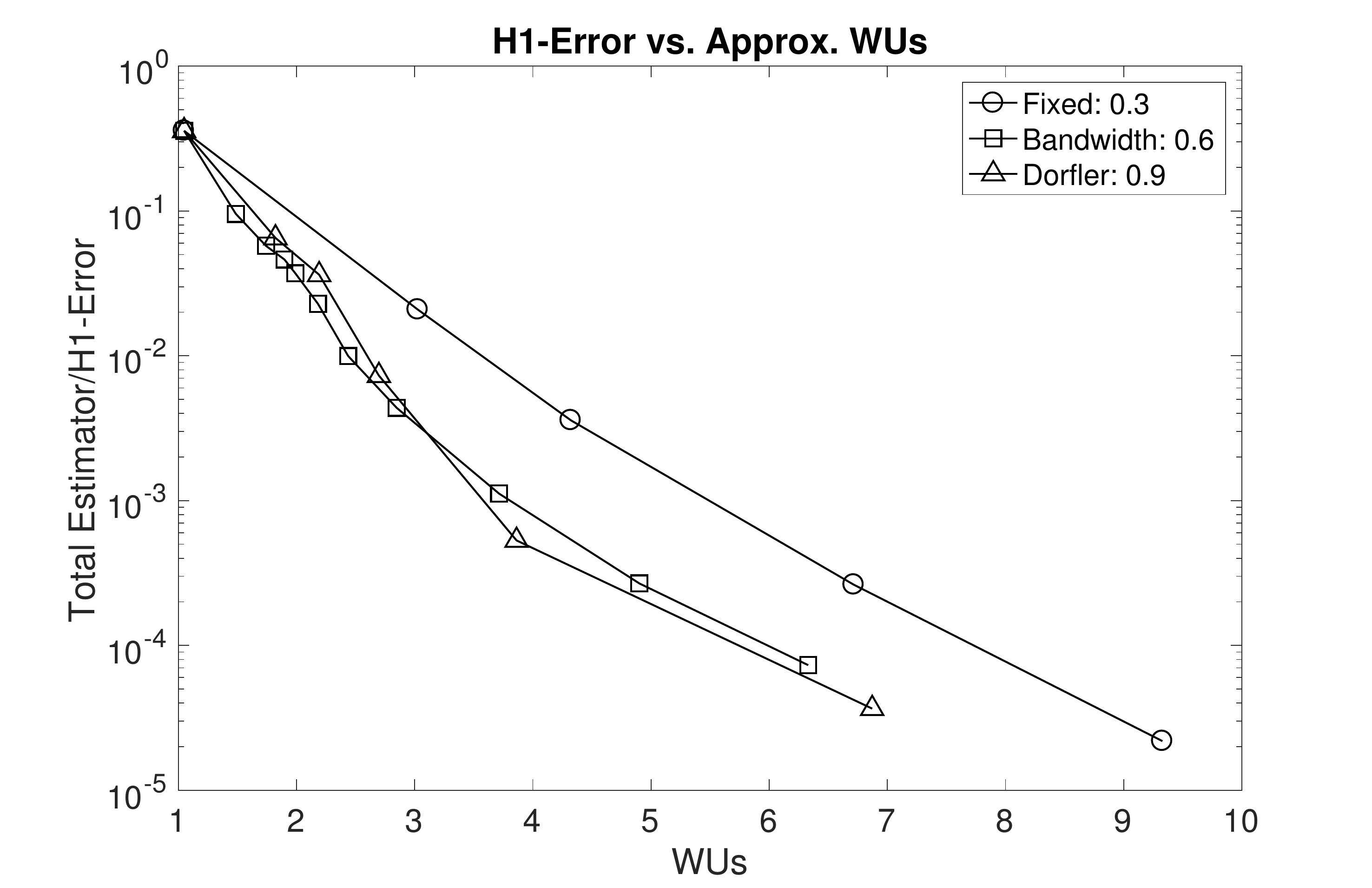}
  \caption{}
  \label{3DErrorAndRefinement:left1}
\end{subfigure}
\begin{subfigure}{0.49 \textwidth}
	\includegraphics[scale=0.232, center]{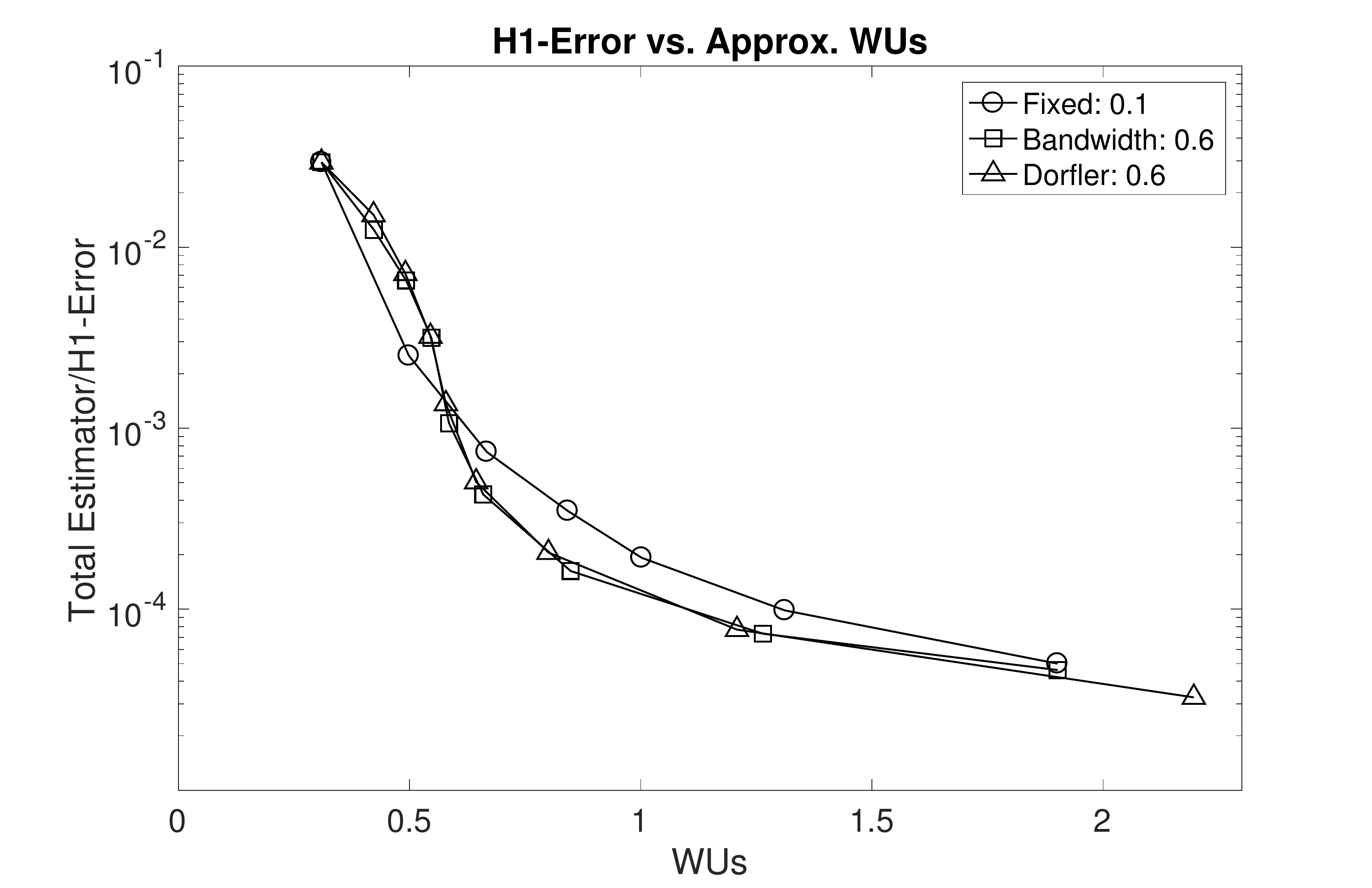}
  \caption{}
  \label{3DErrorAndRefinement:right1}
\end{subfigure}
\caption{\small{
Reduction of $H^1$ approximation error as a function of consumed approximate WUs on each refinement level for the three marking strategies using the penalty (\subref{3DErrorAndRefinement:left1}) and Lagrange multiplier (\subref{3DErrorAndRefinement:right1}) methods.
}}
\label{3DErrorAndRefinement}
\end{figure}

\subsection{2D Flexoelectric Results}

The final set of experiments considers a liquid crystal system with a large applied electric field and flexoelectric coupling on a unit-square domain. The non-dimensionalized physical parameters for $5$CB, a common liquid crystal, are used such that $K_1 = 1$, $K_2 = 0.62903$, $K_3 = 1.32258$, $\epsilon_{\perp} = 7$, and $\epsilon_a = 11.5$. The non-dimensionalized free space permittivity is $\epsilon_0 = 1.42809$, and the flexoelectric constants are $e_s = 1.5$ and $e_b = -1.5$. Finally, the penalty parameter is $\zeta = 10^5$. Each of the simulations begins on a $16 \times 16$ mesh followed by $5$ levels of uniform refinement or $6$ levels of AMR. Uniform boundary conditions are applied for the director field, fixing $\director = (0, 0, 1)^T$. The electric potential is set to zero along the boundary except along $y = 1.0$ where an approximate square function is used such that $\phi$ rises to $1.5$ on roughly the middle-third of the edge. This produces a large electric field with a sharp transition near the top boundary.

\begin{figure}[h!]
\begin{subfigure}{0.32 \textwidth}
	\includegraphics[scale=0.2, center]{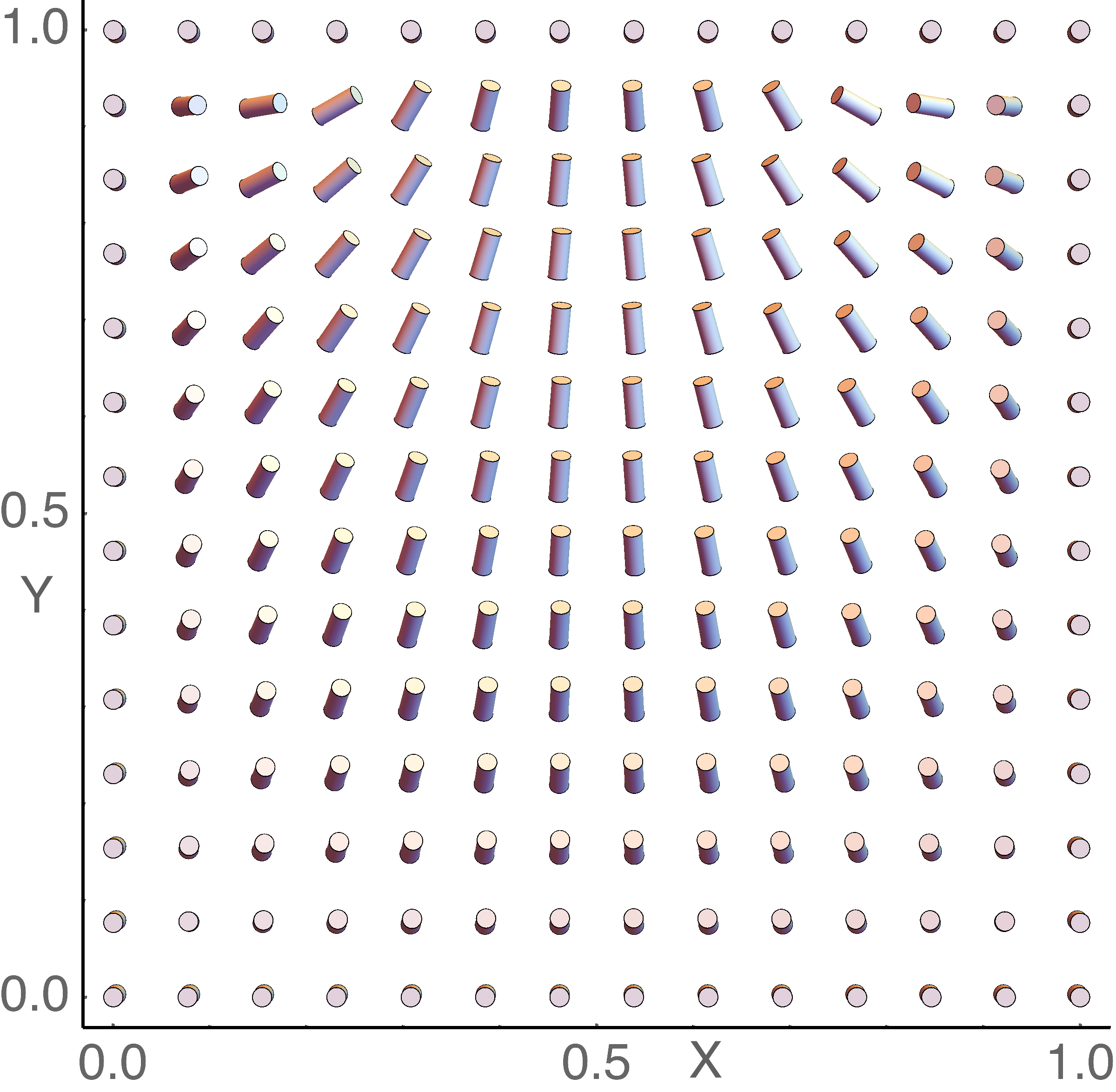} \\ \vspace{-0.0cm}
  \caption{}
  \label{2DFlexoRefinement:left1}
\end{subfigure}
\begin{subfigure}{0.33 \textwidth}
	\includegraphics[scale=0.1315, center]{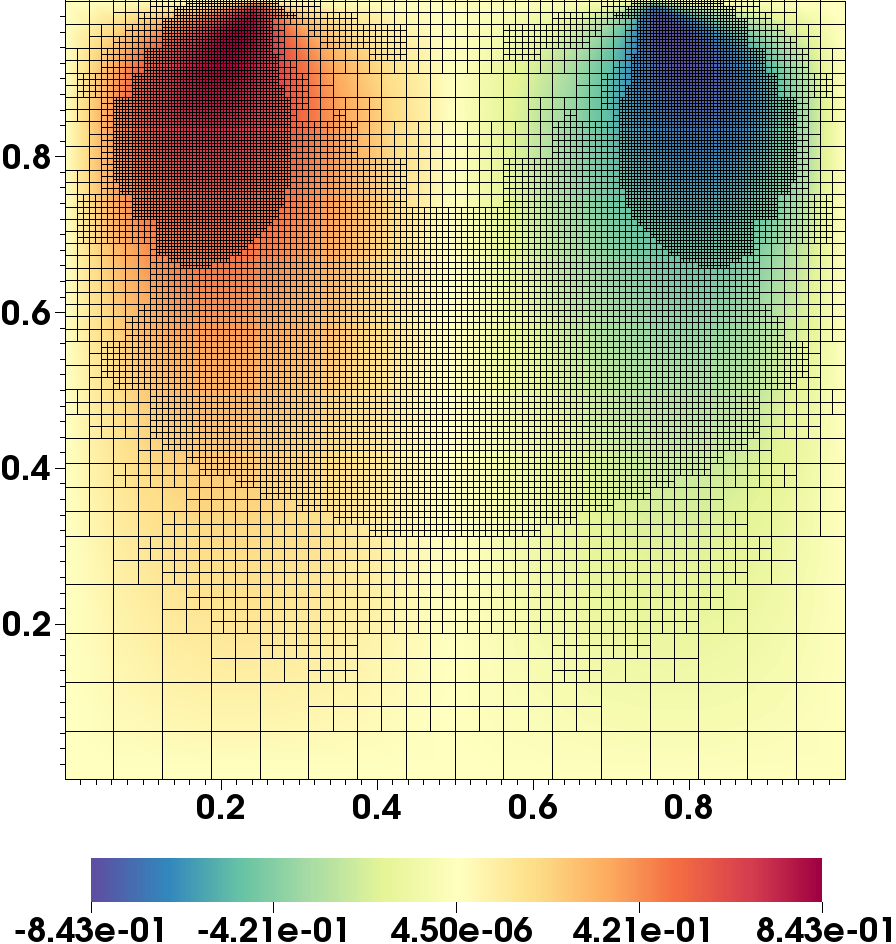}
  \caption{}
  \label{2DFlexoRefinement:center1}
\end{subfigure}
\begin{subfigure}{0.33 \textwidth}
	\includegraphics[scale=0.12, center]{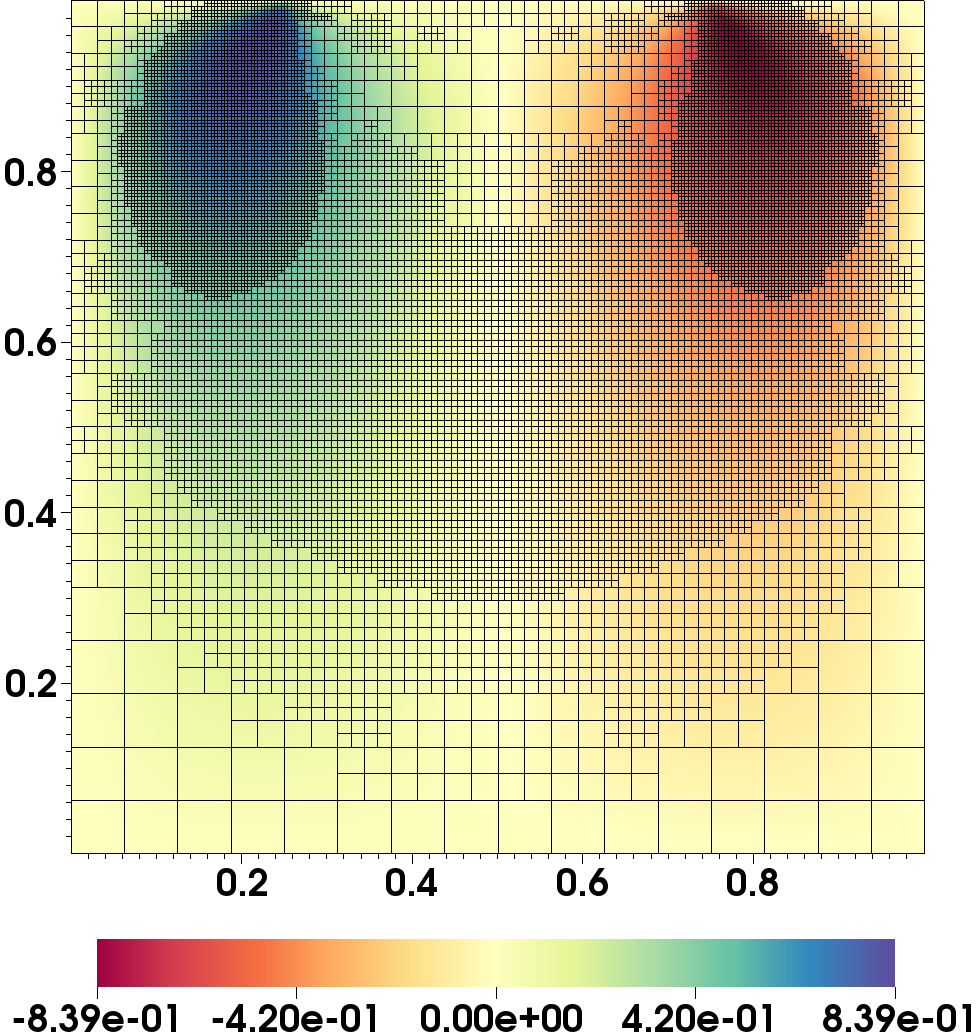}
  \caption{}
  \label{2DFlexoRefinement:right1}
\end{subfigure}
\caption{\small{
(\subref{2DFlexoRefinement:left1})  Fine-mesh computed solution (restricted for visualization) using the penalty method and Dorfler AMR. (\subref{2DFlexoRefinement:center1}, \subref{2DFlexoRefinement:right1}) Resulting mesh patterns after four levels of refinement for the penalty and Lagrange multiplier formulations, respectively, overlaid on the value of $n_1$.
}}
 \label{2DFlexoRefinement}
\end{figure}

The effects of the large electric field are seen in Figure \ref{2DFlexoRefinement}\subref{2DFlexoRefinement:left1}, which shows the computed solution on the finest mesh for the penalty method with Dorfler AMR and $\nu = 0.9$. In response to the field, the director deforms to align with the field lines, even near the boundary where elastic resistance is strongest. The regions surrounding the rapid transitions in the electric potential contain the most difficult to resolve physics and the largest free energy contributions, which suggests that a significant portion of the total approximation error will also be present in these areas. In Figures \ref{2DFlexoRefinement}\subref{2DFlexoRefinement:center1} and \ref{2DFlexoRefinement}\subref{2DFlexoRefinement:right1}, the refinement patterns resulting from Dorfler AMR for the penalty and Lagrange multiplier formulations, respectively, clearly emphasize the transition regions.

\begin{table}[h!]
\centering
\resizebox{\textwidth}{!}{
\begin{tabular}{|c|c|c|c|c|c|c|}
\hlinewd{1.3pt}
 & \multicolumn{3}{|c|}{Penalty (Adapt.)} & \multicolumn{3}{|c|}{Penalty (Uniform)} \\
\hlinewd{1.3pt}
 & Pos. Dev. & Neg. Dev. & Gauss & Pos. Dev. & Neg. Dev. & Gauss\\
 \hline
Fine Conf. & $4.460$e-$02$ & $2.747$e-$02$ & $101.794$ & $4.472$e-$02$ & $2.753$e-$02$ & $357.844$ \\
\hline
Fine Energy & \multicolumn{3}{|c|}{$-39.4726$} & \multicolumn{3}{|c|}{$-39.4858$} \\
\hline
Fine DOF & \multicolumn{3}{|c|}{$2,640,004$} & \multicolumn{3}{|c|}{$4,202,500$} \\
\hline
WUs & \multicolumn{3}{|c|}{$2.427$} & \multicolumn{3}{|c|}{$4.764$} \\
\hline
Timing & \multicolumn{3}{|c|}{$2,700s$} & \multicolumn{3}{|c|}{$7,174s$} \\
\hlinewd{1.3pt}
 & \multicolumn{3}{c|}{Lagrangian (Adapt.)} & \multicolumn{3}{|c|}{Lagrangian (Uniform)} \\
\hlinewd{1.3pt}
 & Pos. Dev. & Neg. Dev. & Gauss & Pos. Dev. & Neg. Dev. & Gauss\\
 \hline
Fine Conf. & $2.316$e-$03$ & $2.470$e-$03$ & $96.395$ & $4.820$e-$04$ & $5.236$e-$04$ & $356.615$ \\
\hline
Fine Energy & \multicolumn{3}{|c|}{$-39.341$} & \multicolumn{3}{|c|}{$-39.355$} \\
\hline
Fine DOF & \multicolumn{3}{|c|}{$2,905,015$} & \multicolumn{3}{|c|}{$4,465,669$} \\
\hline
WUs & \multicolumn{3}{|c|}{$2.382$} & \multicolumn{3}{|c|}{$4.601$} \\
\hline
Timing & \multicolumn{3}{|c|}{$3,999s$} & \multicolumn{3}{|c|}{$9,724s$} \\
\hline
\end{tabular}
}
\caption{\small{Statistics associated with the flexoelectric problem comparing solutions computed with AMR and Dorfler marking to those applying uniform refinement. The first row in each table corresponds to the largest director deviations above and below unit-length at the quadrature nodes and the solutions conformance to Gauss' law on the finest mesh.
}}
\label{flexoAMRUniformStats}
\end{table} 

While no analytical solution exists for this problem, there are a number of indicative metrics that enable comparison of computed numerical solutions. Table \ref{flexoAMRUniformStats} presents these statistics contrasting the quality of approximate solutions produced on uniform meshes with those computed through Dorfler AMR with $\nu = 0.9$, as it performed well in the previous experiments. As expected, the AMR experiments for both constraint enforcement formulations compute solutions in considerably less time and consume half the WUs. In the Lagrange multiplier case, the largest observed deviations of $\director$ from unit-length with AMR remain competitive with those resulting from uniform refinement. For the penalty method, the AMR solution actually exhibits slightly tighter unit-length conformance compared to the finest uniform mesh. Furthermore, the solutions computed with AMR have comparable free energies to those found with uniform refinement. Finally, Table \ref{flexoAMRUniformStats} reports each solution's local Gauss' law conformance over the domain, measured as $\sum_{T \in \triangulation} \int_T (\diverg \vec{D})^2  \diff{V}$. As no special consideration or care has been taken to strongly enforce conformance outside of adherence to the first-order optimality conditions, the sharp boundary conditions of the electric potential lead to relatively large values. However, conformance for solutions constructed with AMR are markedly better, implying more accurate capture of the relevant physics.

\section{Conclusion and Future Work} \label{conclusions}

We have discussed a posteriori error estimators for the electrically and flexoelectrically coupled Frank-Oseen models of nematic liquid crystals with the necessary unit-length constraint enforced via a penalty method or a Lagrange multiplier. The theory developed in \cite{Emerson6} was extended to the proposed coupled estimator for the penalty case showing that it provides a reliable estimate of global approximation error and is an efficient indicator of local error. While analogous theory for the estimator associated with the Lagrangian formulation is the subject of current work, numerical results suggested that it is also highly effective in guiding AMR. The estimators are comprised of readily computable, local quantities suitable for use as part of standard cell flagging schemes.

In addition to the theoretical work, the numerical results of \cite{Emerson6} were expanded to consider several configurations with analytical solutions for purely elastic effects on both 2D and 3D domains. The existence of such solutions enabled verification of the theory in \cite{Emerson6} and concrete comparison of three established marking strategies leveraging the estimators. The most consistent approach for the simulations was Dorfler marking. Finally, the results of using the proposed estimators for a flexoelectrically coupled problem with a challenging applied electric field were presented. In all numerical experiments, application of the error estimators for both constraint enforcement formulations provided accurate cell marking and significantly reduced the amount of work necessary to achieve approximation errors equivalent to or better than those of uniform meshes. Furthermore, AMR guided by the estimators led to more uniformly distributed approximation error, which suggests the constructed meshes are nearer to optimal discretizations. Future work will include extending the theoretical framework to demonstrate reliability and efficiency of the error estimator associated with the Lagrange multiplier formulation. Further, an investigation of nonlinear multigrid methods to directly solve the first-order optimality conditions will be undertaken.

\section*{Acknowledgments}

The author would like to thank Professors James Adler and Xiaozhe Hu for their helpful suggestions and guidance.

\bibliographystyle{siam}

\bibliography{MathematicalCitations}

\end{document}